\documentclass[12pt]{article}

\usepackage{tikz,amsthm,amsmath,amstext,amssymb,epsfig,euscript, mathrsfs, dsfont,pspicture,multicol,graphpap,graphics,graphicx,times,enumerate,subfig,sidecap,
wrapfig, color}
\usepackage{enumerate}

\newcommand\res{\hbox{ {\vrule height .22cm}{\leaders\hrule\hskip.2cm} } }
\def\R{{\mathbb{R}}}
\def\H{{\mathscr{H}}}
\def\sm{\setminus}
\def\wt{\widetilde}

\def\ms{\medskip}

\def\1{{\mathds 1}}
\usepackage{hyperref} 
\hypersetup{backref,  pdfpagemode=FullScreen} 
\usepackage{esint}

\def\bB{{\mathbb{B}}}

\def\bZ{{\mathbb{Z}}}

\def\cD{{\mathscr{D}}}

\def\cF{{\mathscr{F}}}
\def\cG{{\mathscr{G}}}
\def\cH{{\mathscr{H}}}

\def\cR{{\mathscr{R}}}
\def\cS{{\mathscr{S}}}

\def\cV{{\mathscr{V}}}

\def\N{{\mathbb{N}}}
\def\W{{\mathbb {W}}}

\def\Tan{\text{Tan}}

\def\Lip{\text{Lip}}

\def\dist{\mathop\mathrm{dist}}

\def\supp{\mathop\mathrm{supp}}

\newcommand{\av}[1]{\left| #1 \right|}

\def\barintgerm_#1{\mathchoice
{\mathop{\vrule width 6pt height 3 pt depth -2.5pt
\kern -8.8pt \intop}\nolimits_{#1}}%
{\mathop{\vrule width 5pt height 3 pt depth -2.6pt
\kern -6.5pt \intop}\nolimits_{#1}}%
{\mathop{\vrule width 5pt height 3 pt depth -2.6pt
\kern -6pt \intop}\nolimits_{#1}}%
{\mathop{\vrule width 5pt height 3 pt depth -2.6pt \kern -6pt
\intop}\nolimits_{#1}}}

\theoremstyle{plain}
\newtheorem{theorem}{Theorem}

\newtheorem{lemma}[theorem]{Lemma}

\theoremstyle{definition}

\newtheorem{definition}[theorem]{Definition}

\numberwithin{equation}{section}
\numberwithin{theorem}{section}

\begin{document}

\title{Wasserstein Distance and the Rectifiability of Doubling Measures: Part II}
\author{Jonas Azzam
Guy David
Tatiana Toro\footnote{The first author was partially supported by NSF RTG grant 0838212.
The second author acknowledges the generous support of
the Institut Universitaire de France, and of the ANR
(programme blanc GEOMETRYA, ANR-12-BS01-0014).
The third author was partially supported by an NSF grants DMS-0856687 and DMS-1361823, a grant from the  
Simons Foundation (\# 228118) and the Robert R. \& Elaine F. Phelps Professorship in 
Mathematics.}}
\maketitle

\abstract 
We study the structure of the support of a doubling measure by analyzing its
self-similarity properties, which we estimate using a variant of the $L^1$ Wasserstein distance.
We show that a measure satisfying certain self-similarity conditions admits a unique (up to 
multiplication by a constant) flat tangent measure at almost every point. This allows 
us to decompose the support into rectifiable pieces of various dimensions.
 
 \ms  
Soit $\mu$ une mesure doublante dans $\R^n$. On introduit deux parties du support 
o\`u $\mu$ a certaines propri\'et\'es d'autosimilarit\'e, que l'on mesure
\`a l'aide d'une variante de la $L^1$-distance de Wasserstein, et on montre qu'en 
chaque point de ces deux parties, toutes les mesures tangentes \`a $\mu$ sont des
multiples d'une mesure plate (la mesure de Lebesgue sur un sous-espace vectoriel).
On utilise ceci pour donner une d\'ecomposition de ces deux parties en ensembles 
rectifiables de dimensions diverses.

\ms
Key words/Mots cl\'es. Rectifiability, Tangent measures, Doubling measures,
Wasserstein distance


\tableofcontents

\section{Introduction}\label{S1}
\subsection{Statement of Results}

In this paper we are concerned with understanding  the rectifiability properties of doubling measures. 
Our ultimate goal is to characterize rectifiable doubling measures. Recently Tolsa provided such a characterization
for 1-rectifiable measures with upper density bounded below (see \cite{To}). His conditions are expressed
in terms of the properties of the density ratios. We are interested in whether self-similarity properties 
yield some sort of rectifiabilty. 
Roughly 
speaking we analyze how the distance between the dilations of a measure appropriately scaled
yield information about the structure of its support. We provide a criteria which ensures that 
the support of a doubling measure can be decomposed as a union of rectifiable pieces
of different dimensions.
In a previous paper \cite{ADTprep},  similar decompositions were obtained by
looking 
at conditions that were expressed in terms of the properties of the \emph{local} distance between the 
measure and flat measures (that is multiples of Hausdorff measures restricted to affine subsets 
of Euclidean space).
In both cases a minor variant of the $L^1$ Wasserstein distance is used to estimate the \emph{good features} of a measure.
 
To present our results we need to define local distances between measures as well as several quantities which describe 
the self similar character of a measure.
In this paper, $\mu$ denotes a Radon measure on $\R^n$
(i.e., a locally finite positive Borel measure),
and $\Sigma = \Sigma_{\mu}$ denotes its support. That is,  
\begin{equation} \label{1.1}
\Sigma = \big\{ x\in \R^n \, ; \,  \mu(B(x,r)) > 0 \text{ for } r > 0\big\},
\end{equation}
where $B(x,r)$ denotes the open ball centered at $x$ and with radius $r$.
We say that $\mu$ is \textit{doubling} when there is a constant $C_\delta>0$ for which
\begin{equation}\label{1.2}  
\mu (B(x,2r)) \leq C_\delta \, \mu(B(x,r)) \mbox{ for all $x\in \Sigma$ and } r>0.
\end{equation}
Let $\bB = B(0,1)$ denote the unit ball in $\R^n$. For $M \geq 0$, denote by
$\Lip_M(\bB)$ the set of functions $\psi : \R^n \to \R$ that are 
$M$-Lipschitz, i.e., such that 
\begin{equation}\label{1.3}  
|\psi(x)-\psi(y)| \leq M |x-y| \ \text{ for } x,y\in \R^n,
\end{equation}
and for which 
\begin{equation}\label{1.4}  
\psi(x)=0 \ \text{ for } x\in \R^n \sm \bB.
\end{equation}

\begin{definition} \label{t1.1}
Let $\mu$ and $\nu$ be measures on $\R^n$, whose restrictions to
$\bB:= B(0,1)$ are probability measures. We set
\begin{equation} \label{1.5}
\W_1(\mu,\nu):=\sup_{\psi \in \Lip_1(\bB)}\av{\int\psi d\mu-\int\psi d\nu}.
\end{equation}
\end{definition}

Thus $\W_1(\mu,\nu)$ only measures some distance between the restrictions 
to $\bB$ of $\mu$ and $\nu$.
This quantity is similar to the usual $L^1$-Wasserstein distance, which by the Kantorovich duality theorem 
has the same definition as $\W_1$ except that the infimum ranges over 
all $1$-Lipschitz nonnegative functions in $\bB$. 
Note that $\W_1$ has appeared before in the study of rectifiability of measures;
see for example \cite{Preiss87}, \cite{Tolsa-uniform-rectifiability}, 
\cite{Tolsa-mass-transport}, and \cite{ADTprep}.
In Section \ref{S5}, we replace $\W_1$ with a smoother version of local distance $\W_\varphi$ 
which is easier to manipulate.
Lemmas \ref{t5.1} and \ref{t5.3} state that $\W_1$ and $\W_\varphi$ are essentially comparable.
We refer to \cite{Villani} for a detailed introduction to Wasserstein distances and their properties. 

\ms

To estimate the self-similarity properties of $\mu$ we use 
several groups of affine transformations of $\R^n$.
Denote by $\cR$ the group of affine isometries of $\R^n$ (i.e., compositions of 
translations, rotations, and symmetries).
Then let $\cG$ denote the group of similar affine transformations, defined by
\begin{equation} \label{1.6}
\cG = \big\{ \lambda R \, ; \, \lambda > 0 \text{ and } R \in \cR \big\}.
\end{equation}
For $G \in \cG$, we denote by $\lambda(G)$ the unique positive number such that
$G = \lambda(G) R$ for some $R\in \cR$. 
Denote by $\cD$ the group of translations and dilations:
\begin{equation} \label{1.7}
\cD = \big\{ \lambda I + a \, ; \, \lambda > 0 \text{ and } a \in \R^n \big\},
\end{equation}
where $I$ denotes the identity on $\R^n$.

The transformations that map a given $x\in \R^n$
to the origin, are denoted by
\begin{equation} \label{1.8}
\cG(x) = \big\{ G \in \cG ; \, G(x) = 0 \big\}
\ \text{ and } \ 
\cD(x)  = \big\{ D \in \cD ; \, D(x) = 0 \big\}.
\end{equation}

To each $G\in \cG$, we associate the measure $\mu^G = G_\sharp \mu$,
which is defined by
\begin{equation} \label{1.9}
\mu^G(A) = \mu(G^{-1}(A))
\ \text{ for every Borel set } A \subset \R^n. 
\end{equation}
When $G \in \cG(x)$ for some $x\in \Sigma$, be may normalize $\mu^G$
and set
\begin{equation} \label{1.10}
\mu^G_0 = {\mu^G \over \mu^G(\bB)} = {\mu^G \over \mu(G^{-1}(\bB))}
= {\mu^G \over \mu(B(x,\lambda(G)^{-1}))}
\end{equation}
because $\mu(B(x,\lambda(G)^{-1})) > 0$. This normalization is needed
if we want to compute $\W_1$-distances. 

A special case of this is when
$G = T_{x,r}$, the element of $\cD$ that maps $B(x,r)$ to $\bB$; then
$\mu^G$ and $\mu^G_0$ are denoted by $\mu^{x,r}$ and $\mu^{x,r}_0$
respectively. That is,
\begin{equation} \label{1.11}
\mu^{x,r}(A) =  \mu(x + rA) \text{ for } A \subset \R^n,
\ \text{ and } \ 
\mu^{x,r}_0 = {\mu^{x,r} \over \mu(B(x,r))}.
\end{equation}

To measure the self-similar nature of $\mu$ we introduce two quantities $\alpha_{\cG}$ and 
$\alpha_{\cD}$.
We fix two parameters $1 < \lambda_1 < \lambda_2 < \infty$.
Set
\begin{equation} \label{1.12}
\cG(x,r) =  \big\{ G\in \cG(x) \, ; \, \lambda_1 r \leq \lambda(G)^{-1} \leq \lambda_2 r 
\big\}
\end{equation}
and then
\begin{equation} \label{1.13}
\alpha_{\cG}(x,r) = \inf\big\{ \W_1(\mu_0^{G},\mu_0^{x,r}) \, ; \, 
G \in \cG(x,r) \big\}.
\end{equation}
Thus, if $\alpha_{\cG}(x,r)$ is small, this means that in $\bB$, 
$\mu^{x,r}_0$ is close to some measure $\mu^G_0$,
obtained via a transformation $G$ that contracts more than $T_{x,r}$ and possibly rotates as well. 
After composition with $T_{x,r}^{-1}$, the fact that $\alpha_{\cG}(x,r)$ is small
can be interpreted as saying that 
in $B(x,r)$, $\mu$ is quite close to the measure $a \mu^{G'}$, 
where $G' = T_{x,r}^{-1} \circ G$ is a contracting element
of $\cG$ that fixes $x$ and $a >0$ is a 
normalizing constant. It is important to note that even though we allow some
flexibility in the choice of $G$ and $G'$, we demand that $G'(x) = x$. This is
the reason why the usual fractal measures do not satisfy the conditions below.

We also use the analogue of 
$\alpha_{\cG}(x,r)$ for the smaller group $\cD$.That is, set
\begin{equation} \label{1.14}
\cD(x,r) =  \big\{ D\in \cD(x) \, ; \, 
\lambda_1 r \leq \lambda(G)^{-1} \leq \lambda_2 r \big\}
\end{equation}
and 
\begin{equation} \label{1.15}
\alpha_{\cD}(x,r) = \inf\big\{ \W_1(\mu_0^{G},\mu_0^{x,r}) \, ; \, 
G \in \cD(x,r) \big\}.
\end{equation}
For $\alpha_{\cD}(x,r)$ we only compare $\mu$ with its image by some dilation 
centered at $x$. Obviously $\alpha_{\cD}(x,r) \geq \alpha_{\cG}(x,r)$. Thus conditions
on $\alpha_{\cD}(x,r)$ are more restrictive than those on $\alpha_{\cG}(x,r)$.

Our goal is to get a control on the part of $\Sigma$ (see \eqref{1.1}) where either $\alpha_{\cG}(x,r)$,
or $\alpha_{\cD}(x,r)$, are sufficiently small. More precisely, we want to control the sets
\begin{equation} \label{1.16}
\Sigma_1 = \big\{ x\in \Sigma \, ; \, 
\int_0^1 \alpha_{\cD}(x,r) {dr \over r } < \infty \big\}
\end{equation}
and 
\begin{equation} \label{1.17}
\Sigma_2 = \big\{ x\in \Sigma \, ; \, 
\int_0^1 \alpha_{\cG}(x,r) {\log(1/r) dr \over r}  < \infty \big\}.
\end{equation}

\ms
\begin{theorem}\label{t1.2}
Let $\mu$ be a doubling measure on $\R^n$, and denote by $\Sigma$
its support. Let $1 < \lambda_1 < \lambda_2 < \infty$ be given,
and define the sets $\Sigma_1$ and $\Sigma_2$ as above. 
Then there are sets $\cS_{0},...,\cS_{n} \subset \Sigma$, such that 
\begin{equation} \label{1.18}
\mu\Big( (\Sigma_1 \cup \Sigma_2) \sm \Big(\bigcup_{d=0}^{n} \cS_{d} \Big)\Big) =0,
\end{equation}
and moreover 
\begin{itemize}
\item $\cS_{0}$ is the set of points where $\Sigma$ has an atom; it is at most countable, and
every point of $\cS_{0}$ is an isolated point of $\Sigma$. 
\item For $1 \leq d \leq n$, if $x\in \cS_d$, there exists a 
$d$-dimensional vector space $V_x$ such that all the tangent measures to $\mu$
at $x$ (defined below) are multiples of the Lebesgue measure on $V_x$.
\item For $1 \leq d \leq n$, $\cS_{d}$ is $d$-rectifiable, and it can be covered by a countable family
of Lipschitz graphs of dimension $d$. 
\end{itemize}
\end{theorem}

\ms
We may see Theorem \ref{1.2} as a structural decomposition of the good parts of $\Sigma$.
Tangent measures will play an important role in proof and the definition of the $\cS_{d}$.
Recall that the set of {\it tangent measures} to $\mu$ at $x$, which will be denoted by
$\Tan(\mu,x)$, is the set of non-zero Radon measures $\sigma$ for which 
there exist a sequence $\{ r_k \}$, with $\lim_{k \to \infty} r_k = 0$,
and a sequence $\{ a_k \}$ of nonnegative numbers, such that
\begin{equation} \label{1.19}
\sigma \text{ is the weak limit of the measures } a_k \mu^{x,r_k}, 
\end{equation}
where the $\mu^{x,r_k}$ are as in \eqref{1.11}. That is, for every continuous function $f$ with compact support,
\begin{equation} \label{1.20}
\int f d\sigma = \lim_{k \to \infty} a_k \int f d\mu^{x,r_k}.
\end{equation}

Note that since here $\mu$ is doubling, $\Tan(\mu,x)$ is not empty 
(see for instance the proof of Lemma~2.1 in \cite{ADTprep}).
Furthermore if  $\mu$ satisfies \eqref{1.2} and $\sigma\in \Tan(\mu,x)$ then $\sigma$ is also doubling 
with a constant at most $C_\delta^2$.
A priori $\Tan(\mu,x)$ may be large. Nevertheless Theorem \ref{t1.2} ensures that for $x\in \cS_d$,  $\Tan(\mu,x)$
is of dimension~$1$.

A Lipschitz graph of dimension $d$ is a set $\Gamma_A$ 
such that 
$$
\Gamma_A = \big\{ x+A(x) \, ; \, x\in V \big\},
$$
where $V$ is a vector space of dimension $d$, $A : V \to V^\perp$ is a Lipschitz map and 
$V^\perp$ denotes the $(n-d)$-dimensional vector space perpendicular to $V$.
In the statement of Theorem \ref{t1.2}, $\cS_d$ can be covered by Lipschitz graphs
where the corresponding function $A$ has Lipschitz constant less than $\varepsilon$, 
where $\varepsilon > 0$ is any small number given in advance.
Note that this yields that $\cS_d$ is $d$-rectifiable while providing additional information in the sense that 
$\cS_d$ is completely covered by Lipschitz graphs not simply up to a set of $\H^d$-measure zero.
Let us make a few more remarks on Theorem \ref{t1.2} and its proof. 
The advantage of using the quantities $\alpha_\cG$ and $\alpha_\cD$ is that they yield information 
not only about the geometry of the support but also about how the measure is distributed on it. 
The decomposition of $\Sigma_1 \cup \Sigma_2$ into pieces of different dimensions is possible once we prove 
that for $\mu$-almost every point $x\in \Sigma_1 \cap \Sigma_2$,
$\Tan(\mu,x)$ is entirely composed of flat measures of a same dimension depending on $x$. 
Recall that {\it flat measures} are multiples of Lebesgue measures on vector
subspaces of $\R^n$; that is, for each integer $d\in [0,n]$, set 
\begin{equation} \label{1.21}
\cF_d = \big\{ c \H^d\res{ V} \, ; \, c \geq 0 \text{ and } V\in G(d,n) \big\},
\end{equation}
where $\H^d$ denotes the $d$-dimensional Hausdorff measure 
(see \cite{Mattila} or \cite{Federer})
and $G(d,n)$ is the set of $d$-planes in $\R^n$.
The set of flat measures is $\cF = \bigcup_{0 \leq d \leq n} \cF_d$.

It is natural to use the self-similarity properties of $\mu$ to get information
on the structure of $\Sigma$, as in Theorem \ref{t1.2}. In particular the 
numbers $\alpha_{\cG}(x,r)$ provide an intrinsic way to measure the regularity of $\mu$.
We contrast this approach with the one taken in \cite{ADTprep} where we 
were interested on the local approximation of the measure by flat measures.

As we shall see in Section \ref{S8}, the additional logarithm in \eqref{1.17} is used
to sum a series which allows us to control the density of $\mu$ on most of $\Sigma_2$.
It may well be an artifact of the proof.

To prove Theorem \ref{t1.2} we need to find a set $\Sigma_0$ which covers almost all $\Sigma_1\cup\Sigma_2$
and such that all tangents to $\mu$ at points in $\Sigma_0$ are flat. To accomplish this 
we define an average analogue of the numbers $\alpha_{\cG}(x,r)$ by
\begin{equation} \label{1.23}
\alpha_{\cG}^\ast(x,r) 
= \fint_{B(x,r)} \fint_{r}^{2r} \alpha_{\cG}(y,t) d\mu(y)dt
\end{equation}
(where $\fint$ is our notation for an average)
and consider the set
\begin{equation} \label{1.24}
\Sigma_0 = \big\{ x\in \Sigma \, ; \, \lim_{r \to 0} \alpha_{\cG}^\ast(x,r) = 0 \big\}.
\end{equation}

\ms
\begin{theorem}\label{t1.3}
Let $\mu$ be a doubling measure on $\R^n$, and denote by $\Sigma$
its support. Let $1 < \lambda_1 < \lambda_2 < \infty$ be given,
and define the functions $\alpha_{\cG}(x,r)$ and $\alpha_{\cG}^\ast(x,r)$
and the set $\Sigma_0$ as above (see \eqref{1.13}, \eqref{1.23}, and \eqref{1.24}). 
Then 
\begin{equation} \label{1.25}
\Tan(\mu,x) \subset \cF \ \text{ for every } x\in \Sigma_0.
\end{equation} 
\end{theorem}

\ms
A consequence of \eqref{1.25} and the fact that elements of different
$\cF_d$ are far away from each other is that for each $x\in \Sigma_0$,
there is an integer $d \in [0,n]$ such that 
\begin{equation} \label{1.26}
\Tan(\mu,x) \subset \cF_d.
\end{equation} 
This is not too hard to prove. In the case of 
Theorem \ref{t1.2}, we obtain more than \eqref{1.26} directly, thus we omit the proof of this fact.

\ms
To deduce Theorem  \ref{t1.2} from
Theorem~\ref{t1.3}, we shall first check that 
\begin{equation} \label{1.22}
\mu\big((\Sigma_1 \cup \Sigma_2) \sm \Sigma_0 \big) = 0.
\end{equation}
This uses standard techniques from measure theory including 
the Lebesgue density theorem.
Then we show that for each $x\in \Sigma_0 \cap (\Sigma_1 \cup \Sigma_2)$,
\begin{equation} \label{1.27}
\Tan(\mu,x) = \big\{ c \sigma \, ; \, c >0 \big\}
\ \text{ for some } \sigma \in \cF.
\end{equation}

Let us now say a few words about the definition of the $\cS_d$. Set
\begin{equation} \label{1.28}
\cS_d = \big\{ x\in \Sigma_0 \cap (\Sigma_1 \cup \Sigma_2) \, ; \,
\Tan(\mu,x) \subset \cF_d \big\};
\end{equation} 
these sets are disjoint, and by \eqref{1.27} or \eqref{1.26}
\begin{equation} \label{1.29}
\Sigma_0 \cap (\Sigma_1 \cup \Sigma_2) = \bigcup_{d=0}^n \cS_d.
\end{equation}
The special set $\cS_0$ is easily dealt with at the beginning of Section \ref{S7}, 
and Theorem~\ref{t1.2} follows as soon as we prove that for $d \geq 1$,
\begin{eqnarray} \label{1.30}
&&\text{$\cS_d$ can be covered by a countable collection}
\\
&&\hskip4cm
\text{of Lipschitz graphs of dimension $d$.} 
\nonumber
\end{eqnarray}

The fact that information on the tangent measures may imply rectifiability properties
for the measure is much better understood since D. Preiss \cite{Preiss87}
showed that if $\mu$ is a Radon measure, not necessarily doubling, such that 
for $\mu$-almost every $x\in \R^n$, the $d$-density 
$\lim_{r\rightarrow 0}\frac{\mu(B(x,r))}{r^{d}}$ exists and is positive and finite,
then $\R^n$ may be covered, up to a set of $\mu$-measure zero, by a
countable collection of $d$-dimensional Lipschitz graphs. 
He deduced this from the hypothesis on the $d$-density and the
fact that at $\mu$-almost every $x\in \Sigma$, $\Tan(\mu,x) \subset \cF$.
In our case, we are unable to use \cite{Preiss87} because we are not given
any information on the density of $\mu$.

We shall use the fact that since $\mu$
is doubling, \eqref{1.27} implies the existence of a tangent $d$-plane to $\Sigma$ at 
$x$, and then \eqref{1.30} for the set where \eqref{1.27} holds.
We include a proof of these simple observations in Section \ref{S7}.

To prove \eqref{1.27}, we shall use the numbers $\alpha_{\cD}(x,r)$ and $\alpha_{\cG}(x,r)$ 
to control the variations of the measures $\mu^{x,r}$ on $\Sigma_1$ and $\Sigma_2$.
Eventually we compare them to the tangent measures.

For points of $\Sigma_1$ we use the triangle inequality
and the summability of the $\alpha_\cD(x,r)$, to show that the
distance between two different tangent measures at $x$ is controlled by
integrals that tend to $0$. To deal with some of the technical complications that 
arise with the distance $\W_1$, we shall introduce in Section \ref{S5} a smoother variant $\W_\varphi$
of the Wasserstein distance, study it briefly, and then use it in Section \ref{S6} to prove \eqref{1.27}
on $\Sigma_1 \cap \Sigma_0$. 

For points of $\Sigma_2$, we'll use the bounds on the numbers 
$\alpha_\cG(x,r)$ to compute the $\W_\varphi$-distance between the 
$\mu^{x,r}$ and the tangent measures. This time we can only work
modulo rotations, but this is enough to control the $\W_1$-distance
from the $\mu_0^{x,r}$ to flat measures, and apply 
Theorem~1.5 in \cite{ADTprep}. This yields additional information on $\Sigma_2\cap\Sigma_0$.
In particular, it guarantees that on the sets $\Sigma_2 \cap \cS_d$,
 $\mu$ is absolutely
continuous with respect to the Hausdorff measure $\H^d$, with
a density that can be computed from the measure of balls, and that
some local mutual absolute continuity of $\mu$ and 
$\H^d_{| \Sigma_2 \cap \cS_d}$ holds. See near \eqref{8.19}
for a statement, and the rest of Section \ref{S8} for the proof.

\ms
There is a significant difference between \eqref{1.25} (or even \eqref{1.26})
and the stronger \eqref{1.27}. For instance, let $\Sigma$ be an asymptotically flat 
snowflake in $\R^2$, constructed in the usual way but with angles that tend slowly to $0$. 
Put on $\Sigma$ the natural measure $\mu$, coming from the parameterization of 
$\Sigma$ (see \cite{DKT}).  In this case
for $\mu$-almost every $x\in \Sigma$, $\Tan(\mu,x)=\cF_1$. Of course $\Sigma$ is not rectifiable,
and Theorem \ref{t1.2} says that $\Sigma_1 \cup \Sigma_2$ is $\mu$-negligible.

\ms
The definitions \eqref{1.23} and \eqref{1.24}
ensure that for $x\in \Sigma_0$, the local self-similarity
character of $\mu$ improves as the balls get smaller and smaller, 
which yields self-similar tangent measures
at all point of $\Sigma_0$. That is, we show that if $x\in\Sigma_0$,
$\sigma \in \Tan(\mu,x)$ and $y$ lies in the support of $\sigma$,
there is a transformation $G\in \cG$ such that
$\lambda_1 \leq \lambda(H)^{-1} \leq \lambda_2$, $H(y)=y$,
and $H_\sharp \sigma = c \sigma$ for some $c> 0$.
See Lemma~\ref{t3.1}, in Section \ref{S3}. 
Once we prove this, showing that $\sigma$ is flat is mostly a matter of playing
with the invariance properties of the support and the measure; see Section \ref{S4}.

 \subsection{Acknowledgements}

The authors are grateful to Alessio Figalli and Xavier Tolsa for helpful discussions. 
Parts of this work were done while the first author was visiting IPAM and while the  
second author was visiting the University of Washington.

\section{Control of the averages $\alpha_\cG$}
\label{S2}

The main goal of this section is to prove \eqref{1.22}.
To this effect define the set
\begin{equation} \label{2.1}
\Sigma_3 = \big\{ x\in \Sigma \, ; \, 
\int_0^1 \alpha_{\cG}(x,r) {dr \over r}  < \infty \big\}.
\end{equation}
Notice that $\Sigma_1 \cup \Sigma_2 \subset \Sigma_3$, by \eqref{1.16}, \eqref{1.17}, 
and because $\alpha_{\cG}(x,r) \leq \alpha_{\cD}(x,r)$.
Thus \eqref{1.22} follows once we prove that
\begin{equation} \label{2.2}
\mu(\Sigma_3 \sm \Sigma_0) = 0.
\end{equation}
For $N > 0$ large and $k \geq 2$, let
\begin{equation} \label{2.3}
\Sigma_3(N) = \big\{ x\in \Sigma_3 \cap B(0,N) \, ; \, 
\int_0^1 \alpha_{\cG}(x,r) {dr \over r}  \leq N  \big\}
\end{equation}
and
\begin{equation} \label{2.4}
\varepsilon_k = \int_{\Sigma_3(N)} \int_{ 2^{-k}}^{2^{-k+2}}
\alpha_{\cG}(y,r) {d\mu(y) dr \over r}.
\end{equation}
Then
\begin{equation} \label{2.5}
\sum_{k \geq 2}\varepsilon_k \leq 2 \int_{ \Sigma_3(N)} \int_0^1
\alpha_{\cG}(x,r) {d\mu(x) dr \over r} \leq 2N \mu(\Sigma_3(N)) < \infty.
\end{equation}
Choose a decreasing sequence $\{ \gamma_k \}$ such that
\begin{equation} \label{2.6}
\lim_{k \to \infty} \gamma_k = 0
\ \text{ and } \ 
\sum_{k \geq 2} \gamma_k^{-1} \varepsilon_k < \infty. 
\end{equation}
For $x\in \Sigma_3(N)$, define auxiliary functions $\alpha_k$ by
\begin{equation} \label{2.7}
\alpha_k(x) = \int_{\Sigma_3(N) \cap B(x,2^{-k+1})} \int_{ 2^{-k}}^{2^{-k+2}}
\alpha_\cG(y,r) {d\mu(y) dr \over r}.
\end{equation}
Consider the bad sets
\begin{equation} \label{2.8}
Z_k = \big\{ x\in \Sigma_3(N) \, ; \, \alpha_k(x) \geq \gamma_k 
\mu(B(x,2^{-k+1})) \big\}.
\end{equation}
Our goal is to show that $Z_k$ is small. Let $X \subset Z_k$ be a maximal subset
whose points lie at distance at least $2^{-k+2}$ from each other. Thus the 
balls $\overline B(x,2^{-k+2})$, $x\in X$, cover $Z_k$, so by \eqref{1.2} and \eqref{2.8}
\begin{eqnarray} \label{2.9}
\mu(Z_k) &\leq & \sum_{x\in X} \mu(\overline B(x,2^{-k+2})) 
\leq C_\delta^2 \sum_{x\in X} \mu(B(x,2^{-k+1})) \nonumber\\
&\leq &C_\delta^2 \gamma_k^{-1} \sum_{x\in X} \alpha_k(x).
\end{eqnarray}
Since the balls $B(x,2^{-k+1})$, $x\in X$, 
are disjoint,
$$
\sum_{x\in X} \alpha_k(x) \leq \varepsilon_k
$$
(compare \eqref{2.4} and \eqref{2.7}); thus
$\mu(Z_k) \leq C_\delta^2 \gamma_k^{-1} \varepsilon_k$.
We are not interested in the precise bound, but merely the fact that
$\sum_k \mu(Z_k) < \infty$,
from which we deduce that if we set $Z^\ast_l = \bigcup_{k \geq l} Z_k$
for $l \geq 2$, then $\lim_{l \to \infty} \mu(Z^\ast_l) = 0$.
Thus for $\mu$-almost every $x\in \Sigma_3(N)$ there is $k_x\in \N$ such that 
\begin{equation} \label{2.10}
x \in \Sigma_3(N) \sm Z_k \hbox{ for } k\ge k_x.
\end{equation}
By the Lebesgue differentiation theorem applied to the doubling measure $\mu$, 
we have that for $\mu$-almost every $x\in \Sigma_3(N)$
\begin{equation} \label{2.11}
\lim_{r\to 0} {\mu(B(x,r) \cap \Sigma \sm \Sigma_3(N)) 
\over  \mu(B(x,r) \cap \Sigma)} = 0;
\end{equation}
see for instance Corollary 2.14 in \cite{Mattila}.

Let $x\in \Sigma_3(N)$ satisfy \eqref{2.10} and \eqref{2.11}; we want to
estimate $\alpha_{\cG}^\ast(x,r)$ for $r$ small. Choose $k$ such that
$2^{-k} \leq r \leq 2^{-k+1}$; then $k\ge k_x$ for $r$ small. Recall from \eqref{1.23} that
\begin{equation} \label{2.12}
\alpha_{\cG}^\ast(x,r) = \mu(B(x,r))^{-1}\int_{y\in B(x,r)} \fint_{r}^{2r} \alpha_{\cG}(y,t) d\mu(y)dt.
\end{equation}
We decompose the domain of integration above into two parts and estimate each one separately.
By \eqref{1.2}, \eqref{2.7}, \eqref{2.10}, and \eqref{2.8}, 
\begin{eqnarray}  \label{2.13}
&\,& \hskip-2cm
\mu(B(x,r))^{-1} \int_{\Sigma_3(N) \cap B(x,r)} 
\fint_{r}^{2r} \alpha_{\cG}(y,t) d\mu(y)dt
\nonumber \\
&\leq& 4\mu(B(x,2^{-k}))^{-1} \int_{\Sigma_3(N) \cap B(x,r)}
\int_{2^{-k}}^{2^{-k+2}}  \alpha_{\cG}(y,t) {d\mu(y)dt \over t}
\nonumber \\
&\leq&  4 C_\delta \mu(B(x,2^{-k+1}))^{-1} \alpha_k(x) \leq 4 C_\delta \gamma_k.
\end{eqnarray}
This term tends to $0$ when $r$ tends to $0$, by \eqref{2.6}.

For the second part, we notice that $\alpha_{\cG}(y,t) \leq 2$
by definition (see \eqref{1.13} and \eqref{1.5}), so
\begin{eqnarray} \label{2.14}
&\,& \hskip-3cm\mu(B(x,r))^{-1} \int_{\Sigma\cap B(x,r) \sm \Sigma_3(N)} 
\fint_{r}^{2r} \alpha_{\cG}(y,t) d\mu(y)dt
\nonumber\\
&\leq& 2 \mu(B(x,r))^{-1} \mu(\Sigma\cap B(x,r) \sm \Sigma_3(N)),
\end{eqnarray}
which tends to $0$ by \eqref{2.11}. Combining \eqref{2.13} and \eqref{2.14} we get that
$$
\lim_{r \to 0} \alpha_{\cG}^\ast(x,r) = 0
\ \text{ for $\mu$-almost every } x\in \Sigma_3(N). 
$$
In other words, $\mu(\Sigma_3(N) \sm \Sigma_0) = 0$
(see \eqref{1.24}); \eqref{2.2} follows easily, and so does \eqref{1.22}.

\section{Tangent measures are self-similar}
\label{S3}

In this section we start the proof of Theorem \ref{t1.3} and
prove the basic self-similarity estimate for tangent measures.

\begin{lemma} \label{t3.1} 
Let $\mu$ be a doubling measure, let $\Sigma$ denote its support,
let $\Sigma_0\subset \Sigma$ be as in \eqref{1.24}, and 
$\sigma \in \Tan(\mu,x)$ be a tangent measure of $\mu$ at a point $x\in \Sigma_0$.
For each $y \in \Xi$, the support of $\sigma$, there exist $H \in \cG$ such that
$H(y) = y$, 
\begin{equation} \label{3.1}
\lambda_1 \leq \lambda(H)^{-1} \leq \lambda_2
\end{equation}
and 
\begin{equation} \label{3.2}
H_\sharp \sigma = c \sigma \ \text{ for some } c > 0.
\end{equation}
\end{lemma}

\ms 
The numbers $\lambda_1$ and $\lambda_2$ are the same as in \eqref{1.12},
the dilation number $\lambda(H)$ is defined below \eqref{1.6},
and $H_\sharp \sigma$, the push forward image of $\sigma$ by $H$, is 
defined as in \eqref{1.9}.

\begin{proof} 
We may assume, without loss of generality, that $x=0$.
Since $\sigma \in \Tan(\mu,x)$ there
are coefficients $a_k \geq 0$ and radii $r_k > 0$, such that
$ \lim_{k \to \infty} r_k = 0$,
and $\sigma$ is the weak limit of the measures $\{ \sigma_k \}$, where
\begin{equation} \label{3.4}
\sigma_k = a_k \mu^{0,r_k} = a_k \mu^{R_k}\hbox{  with } R_k(u) = r_k^{-1} u \ \text{ for } u \in \R^n. 
\end{equation}
Note that $R_k$ maps $B(x,r_k)=B(0,r_{k})$ to $\bB$. See \eqref{1.11} for the definition of $\mu^{x,r}$. 

Let
\begin{equation} \label{3.8}
\alpha_k = \sup\big\{ \alpha_\cG^\ast(0,r) \, ; \, 0 < r < \sqrt{r_k} \big\};
\end{equation}
then since $x\in\Sigma_0$ (see \eqref{1.24})
\begin{equation} \label{3.9}
\lim_{k \to \infty} \alpha_k = 0.
\end{equation}
By \eqref{3.8}, if for $k$ large
\begin{equation} \label{3.7}
r_k < \rho_k < \sqrt{r_k} 
\end{equation}
then $\alpha_\cG^\ast(0,\rho_k) \leq \alpha_k$ for these $k$.
Since $\sigma$ is the weak limit of the $\sigma_k$, for each $y\in\Xi:=\supp \sigma$ we can find points 
$y_k \in \supp(\sigma_k) = r_k^{-1} \Sigma$ such that
\begin{equation} \label{3.12}
\lim_{k \to \infty} |y_k-y| = 0.
\end{equation}
Let $\{\eta_k\}_{k\ge1}$ and $\{\rho_k\}_{k\ge 1}$ be sequences such that \eqref{3.7} holds
for $k$ large, and also
\begin{equation} \label{3.6}
\lim_{k \to \infty} {\rho_k \over r_k} = \infty
\ \text{ and } \ 
\lim_{k \to \infty} \eta_k = 0.
\end{equation}

Consider
\begin{equation} \label{3.13}
A_k =  \fint_{ B(r_k y_k, \eta_k r_k)} \fint_{ \rho_k}^{2\rho_k} 
\alpha_{\cG}(z,t) d\mu(z)dt.
\end{equation}
For $k$ large, the domain of integration $B(r_k y_k, \eta_k r_k)$ is contained 
in $B(0,\rho_k)$ (because $y_k$ tends to $y$, $\eta_k$ tends to $0$,
and $r_k^{-1}\rho_k$ tends to $+\infty$; see \eqref{3.12} and \eqref{3.6}).
Recall that $B(0,\rho_k) \times [\rho_k,2\rho_k]$ is the domain of integration in the definition of
$\alpha_\cG^\ast(0,\rho_k)$ (see \eqref{1.23}); hence for $k$ large
\begin{equation} \label{3.14}
A_k \leq {\mu(B(0,\rho_k)) \over \mu(B(r_k y_k, \eta_k r_k))} \,
\alpha_\cG^\ast(0,\rho_k)
\leq C_\delta^{2 + \log_2(\rho_k/(\eta_k r_k))} \alpha_k
\end{equation}
by \eqref{3.13}, \eqref{1.23}, \eqref{3.7}, \eqref{3.8}, and the doubling property \eqref{1.2}.
Later on, we will choose $\rho_k$ and $\eta_k$, depending on $\alpha_k$, 
so that $A_k$ is still small enough. 

By Chebyshev's inequality there exist
\begin{equation} \label{3.15}
z_k \in \Sigma \cap B(r_k y_k, \eta_k r_k)
\ \text{ and } \ t_k \in [\rho_k,2\rho_k]
\end{equation}
such that
\begin{equation} \label{3.16}
\alpha_{\cG}(z_k,t_k) \leq A_k.
\end{equation}
By the definition of $\alpha_{\cG}$ (see  \eqref{1.13}) there exists 
$G_k \in \cG(z_k,t_k)$ such that 
\begin{equation} \label{3.17}
\W_1(\mu_0^{G_k},\mu_0^{z_k,t_k}) \leq 2 A_k,
\end{equation}
which by \eqref{1.5} means that
\begin{equation} \label{3.18}
\av{\int\psi d\mu_0^{G_k}-\int\psi \mu_0^{z_k,t_k}} \leq 2 A_k 
 \ \text{ for any } \psi \in \Lip_1(\bB).
\end{equation}

Our next goal is to interpret \eqref{3.18} in terms of $\sigma_k$.
Let $T_k \in \cD$ be such that for $u\in \R^n$
\begin{equation} \label{3.19}
T_k(u) = {u-z_k \over t_k} \ \hbox{ and so }\ T_k(B(z_k,t_k)) = \bB.
\end{equation}
The definitions \eqref{1.10} and \eqref{1.11} yield
\begin{equation} \label{3.20}
\mu_0^{G_k} = e_k \mu^{G_k} \ \text{ and } \ 
\mu_0^{z_k,t_k} = e'_k \mu^{z_k,t_k} = e'_k \mu^{T_k},
\end{equation}
where $e_k$ and $e'_k$ come from the normalization, and are given by
\begin{equation} \label{3.21}
e_k = \mu^{G_k}(\bB)^{-1}
\ \text{ and } \ 
e'_k = \mu^{T_k}(\bB)^{-1}.
\end{equation}
Let $\psi$ be any Lipschitz function supported on $\bB$ and set 
$I_k = e_k \int_{G_k(\Sigma)}\psi d\mu^{G_k}$; by \eqref{3.20} and \eqref{1.9}, 
\begin{equation} \label{3.23}
I_k = e_k \int_{G_k(\Sigma)}\psi d\mu^{G_k} 
= e_k \int_\Sigma\psi(G_k(\xi)) d\mu(\xi)
= e_k \int\psi \circ G_k \, d\mu.
\end{equation}

By \eqref{3.4}, $\sigma_k = a_k \mu^{R_k} = a_k (R_k)_\sharp \mu$, hence
$\mu = a_k^{-1} (R_k^{-1})_\sharp \sigma_k$. Thus a similar computation to the one in \eqref{3.23} yields
\begin{equation} \label{3.24}
I_k = e_k a_k^{-1} \int \psi \circ G_k \circ R_k^{-1} d\sigma_k.
\end{equation}

A similar computation, with $G_k$ replaced by $T_k$, yields
\begin{equation} \label{3.25}
I'_k :=  \int\psi \mu_0^{z_k,t_k}= e'_k a_k^{-1} \int \psi \circ T_k \circ R_k^{-1} d\sigma_k.
\end{equation}

We want to apply \eqref{3.24} and \eqref{3.25} to special functions $\psi$.
Let $\varphi$ be a compactly supported 1-Lipschitz function. Let
\begin{equation} \label{3.27}
\psi = \varphi \circ R_k \circ T_k^{-1}.
\end{equation}
Note that $\psi$ is a Lipschitz function with constant less or equal to 
$r_k^{-1}t_k\le 2r_k^{-1}\rho_k$ (see \eqref{3.15}).
If $\varphi$ is supported in $B(0,R)$, for $k$ large enough the support of $\psi$ is contained in 
$$\begin{aligned}
T_k \circ R_k^{-1}(B(0,R)) &= T_k(B(0,r_kR)) = B\big({-z_k \over t_k}, t_k^{-1} r_k R\big)
\\
&\subset B(0, r_kt_k^{-1}(R+\eta_k+|y_k|))\subset B(0,1),
\end{aligned}
$$
 where we have used \eqref{3.15}, \eqref{3.6}, and \eqref{3.12}.
Because of \eqref{3.14}, $\wt \psi = (2r_k^{-1} \rho_k)^{-1}\psi$ is $1$-Lipschitz; 
then \eqref{3.18} applies to $\wt \psi$, and \eqref{3.14} yields
\begin{equation} \label{3.29}
|I'_k - I_k| \leq 4 A_k r_k^{-1} \rho_k \leq 
4 r_k^{-1} \rho_k C_\delta^{2 + \log_2(\rho_k/(\eta_k r_k))} \alpha_k
=: \wt \alpha_k
\end{equation}
where $I_k$ and $I'_k$ are as in \eqref{3.23} and \eqref{3.25} with $\psi$ coming from \eqref{3.27}.
The final identity is the definition
of $\wt \alpha_k$. Notice that even though $\varphi$
does not depend on $k$, $\psi$ does, but this is not an issue.

Note that by \eqref{3.27} and \eqref{3.25}, we have
\begin{equation} \label{3.30}
I'_k = e'_k a_k^{-1} \int \varphi d\sigma_k.
\end{equation}
Similarly, by \eqref{3.24} and \eqref{3.27} we have
\begin{equation} \label{3.31}
I_k = e_k a_k^{-1} \int \varphi \circ R_k \circ T_k^{-1} \circ G_k \circ R_k^{-1} d\sigma_k. 
\end{equation}
Set 
\begin{equation}\label{3.31A}
H_k = R_k \circ T_k^{-1} \circ G_k \circ R_k^{-1}.
\end{equation}
Then $H_k \in \cG$, and by \eqref{3.19}, its dilation factor $\lambda(H_k)$
(which is also the $n$-th root of its Jacobian determinant) is such that
\begin{equation} \label{3.32}
\lambda(H_k)^{-1} = \lambda(T_k) \lambda(G_k)^{-1}
= t_k^{-1} \lambda(G_k)^{-1} \in [\lambda_1, \lambda_2],
\end{equation}
because $G_k \in \cG(z_k,t_k)$, and by
the definition \eqref{1.12}. 
Note that by \eqref{3.4}, the fact that $G_k \in \cG(z_k,t_k)$ 
(see \eqref{1.12} and \eqref{1.8}), and \eqref{3.19}, we have
\begin{eqnarray} \label{3.33}
H_k(r_k^{-1} z_k) &= & R_k \circ T_k^{-1} \circ G_k \circ R_k^{-1}(r_k^{-1}z_k) 
= R_k \circ T_k^{-1} \circ G_k (z_k)\nonumber\\
&=&R_k \circ T_k^{-1}(0)= R_k(z_k)=
r_k^{-1} z_k.
\end{eqnarray}

Notice also that by \eqref{3.15}
$$
|r_k^{-1} z_k - y| \leq |r_k^{-1} z_k - y_k| + |y_k-y|
= r_k^{-1} | z_k - r_k y_k| + |y_k-y| \leq \eta_k + |y_k-y|.
$$
Thus $|r_k^{-1} z_k - y| $ tends to $0$ by \eqref{3.12} and \eqref{3.6}, so
\begin{equation} \label{3.34}
\lim_{k \to \infty} r_k^{-1} z_k = y.
\end{equation}
Combining \eqref{3.32}, \eqref{3.33}, and \eqref{3.34}, we deduce that $H_k$ lies in a compact subset
of $\cG$. Thus we can replace $\{ r_k \}$ by a subsequence for which the 
$H_k$ converge to a limit $H$. In addition,  \eqref{3.32}, \eqref{3.33} and \eqref{3.34}
imply that
\begin{equation} \label{3.35}
\lambda(H)^{-1} \in [\lambda_1, \lambda_2]
\ \text{ and } \ 
H(y) = y.
\end{equation}
Combining \eqref{3.29},  \eqref{3.30}, \eqref{3.31}, and \eqref{3.31A} we 
see that if $\varphi$ is a compactly supported $1$-Lipschitz function, then for $k$ large
\begin{equation} \label{3.36}
\Big| e_k a_k^{-1} \int \varphi \circ H_k \, d\sigma_k 
- e'_k a_k^{-1} \int \varphi d\sigma_k \Big| \leq \wt\alpha_k.
\end{equation}

By \eqref{3.21}, \eqref{1.9}, and \eqref{3.19}, 
\begin{equation} \label{3.37}
e'_k = \mu^{T_k}(\bB)^{-1} = \mu(T_k^{-1}(\bB))^{-1} = \mu(B(z_k,t_k))^{-1}.
\end{equation}
Similarly notice that $G_k(z_k) = 0$ because $G_k \in \cG(z_k,t_k) \subset \cG(z_k)$
(see \eqref{1.12} and \eqref{1.8}); then \eqref{3.21} and \eqref{1.9} yield
\begin{equation} \label{3.37A}
e_k = \mu^{G_k}(\bB)^{-1} = \mu(G_k^{-1}(\bB))^{-1}
= \mu(B(z_k,\lambda(G_k)^{-1}))^{-1}.
\end{equation}
In addition,
\begin{equation} \label{3.37B}
B(z_k,\lambda_1 t_k) \subset G_k^{-1}(\bB) = B(z_k,\lambda(G_k)^{-1})
\subset B(z_k,\lambda_2 t_k)
\end{equation}
because $G_k(z_k) = 0$ and by \eqref{1.12}.
Then \eqref{1.2}, \eqref{3.37}, \eqref{3.37A}, \eqref{3.37B}, and the fact that $\lambda_1>1$ yield
\begin{equation} \label{3.38}
C^{-1} e'_k \leq e_k \leq e'_k
\end{equation}
for some constant $C$ that depends on $C_\delta$ and $\lambda_2$.

To estimate $a_k$, consider a test function $f$ such that $ \1_{\bB}\le f\le \1_{2\bB}$.
By definition of $\sigma$, $\int f d\sigma = \lim_{k \to \infty} \int f d\sigma_k$.
By  \eqref{1.2}, \eqref{3.4} and the definition above \eqref{1.11}, we have
\begin{equation}\label{3.38A}
a_k \mu(B(0,r_k)) = \sigma_k(\bB)\le \int f d\sigma_k 
\le \sigma_k(2\bB) =   a_k \mu(B(0,2r_k)) \le C_\delta a_k \mu(B(0,r_k)).
\end{equation}
Moreover, since $\sigma$ is also doubling (see the remark below \eqref{1.20}), we have that
\begin{equation}\label{3.38B}
\sigma(\bB)\le \int f d\sigma \le \sigma(2\bB)\le C_\delta ^2\sigma(\bB).
\end{equation}

Thus by \eqref{3.38A}, \eqref{3.38B} and the definition of $\sigma$ there exists $C>1$ such that for $k$ large, 
\begin{equation}\label{3.38C}
C^{-1}  a_k \mu(B(0,r_k))\le  \sigma(\bB)\le CC_\delta^2 a_k \mu(B(0,r_k)).
\end{equation}

Recall that $\rho_k \leq t_k \leq 2\rho_k$ by \eqref{3.15}, and 
that $r_k^{-1} z_k$ is bounded by \eqref{3.34}; since $r_k^{-1} \rho_k$ tends to $+\infty$ 
by \eqref{3.6} we get that for $k$ large, $|z_{k}|<Cr_{k}<\rho_{k}\leq t_{k}$, and so
$B(z_k,t_k) \subset B(0, 2t_k) \subset B(0, 4\rho_k)$.
By \eqref{3.37}, since $0 \in \Sigma$, and by \eqref{1.2},
\begin{equation}\label{3.38D}
(e'_k)^{-1}=\mu(B(z_k,t_k)) \le \mu(B(0,4\rho_k))\leq  C_\delta^{3+\log_2(\rho_k/r_k)} \mu(B(0,r_k)).
\end{equation}
Combining \eqref{3.38C} and \eqref{3.38D} we obtain
\begin{equation} \label{3.40}
a_k \leq C C_\delta^{5+\log_2(\rho_k/r_k)} e'_k \cdot \sigma(\bB).
\end{equation}

Return to \eqref{3.36}, set $b_k = e_k/e'_k$, and observe that by \eqref{3.40} 
\begin{equation} \label{3.41}
\Big| b_k \int \varphi \circ H_k \, d\sigma_k - \int \varphi d\sigma_k \Big|
\leq {a_k \wt\alpha_k\over e'_k} 
\leq C C_\delta^{5+\log_2(\rho_k/r_k)} \wt\alpha_k \sigma(\bB).
\end{equation} 

Now we choose $\rho_k$ and $\eta_k$. Denote by
$$
\beta_k = C C_\delta^{5+\log_2(\rho_k/r_k)} \wt\alpha_k
= C C_\delta^{5+\log_2(\rho_k/r_k)} \cdot
4 r_k^{-1} \rho_k C_\delta^{2 + \log_2(\rho_k/(\eta_k r_k))} \alpha_k
$$
the right-hand side of \eqref{3.41} (see \eqref{3.29}). 
Since $\alpha_k$ tends to $0$ by \eqref{3.9},
we can choose $\rho_k$ and $\eta_k$ so that the constraints \eqref{3.6} 
and \eqref{3.7} hold, but the convergence in \eqref{3.6} is slow enough so
that
\begin{equation} \label{3.42}
\lim_{k \to \infty}  \beta_k = 0. 
\end{equation}
Recall from \eqref{3.38} that $C^{-1} \leq b_k \leq 1$, hence modulo passing to a subsequence
(which we relabel) we can guarantee that
$\lim_{k \to \infty} b_k = b > 0$. Letting $k$ tend to infinity in \eqref{3.41} we obtain
\begin{equation} \label{3.42A}
\lim_{k \to \infty} \int \varphi d\sigma_k = \int \varphi d\sigma.
\end{equation}
Since $\varphi$ is Lipschitz and compactly supported, so is
$\varphi\circ H$ and 
\begin{equation} \label{3.42B}
\lim_{k \to \infty} \int \varphi\circ H \, d\sigma_k 
= \int \varphi\circ H \, d\sigma
\end{equation}
because $\sigma$ is the weak limit of the $\sigma_k$.
Note that there is also a ball $B$ such that for $k$ large
$\varphi\circ H(x) = \varphi\circ H_k(x) = 0$ for
$x\in \R^n \sm B$; then
\begin{align}\label{3.42C}
\int |\varphi\circ H - \varphi\circ H_k| \, d\sigma_k 
&\leq ||\varphi||_{lip} \int_B |H - H_k| \, d\sigma_k \nonumber
\\
&\leq ||\varphi||_{lip} ||H - H_k||_{L^\infty(B)}\sigma_k(B)\nonumber
\\
&\leq 2 ||\varphi||_{lip} ||H - H_k||_{L^\infty(B)}\sigma(2B).
\end{align}
Thus
\begin{equation}\label{3.42D}
\lim_{k\rightarrow \infty}\left|\int  \varphi\circ H_k\, 
d\sigma_k -\int  \varphi\circ H\, d\sigma_k \right|=0.
\end{equation}
Combining \eqref{3.41}, \eqref{3.42}, \eqref{3.42A}, \eqref{3.42B}, \eqref{3.42C} and \eqref{3.42D}
we obtain that for any $1$-Lipschitz function $\varphi$ with compact support,
\begin{equation} \label{3.43}
b \int \varphi \circ H \, d\sigma = \int \varphi \,d\sigma.
\end{equation}
Since the Radon measure $\sigma$ is regular, \eqref{3.43} also 
holds for characteristic functions of Borel sets. Hence $b H_\sharp \sigma = \sigma$.

Recall that $\lambda(H)^{-1} \in [\lambda_1, \lambda_2]$
and $H(y) = y$ by \eqref{3.35}; thus the conclusion of Lemma \ref{t3.1}
hold, with $c=b^{-1}$.
\end{proof}

\section{Self-similar measures are flat}
\label{S4}

In this section we complete the proof of Theorem \ref{t1.3}. 
Using the notation in Section \ref{S3}, our goal is to show that 
if  $\sigma \in \Tan(\mu,x_0)$, where $x_0\in \Sigma_0$ (see \eqref{1.24})
then $\sigma$ is a flat measure. Lemma~\ref{t3.1} guarantees that for each $y\in \Xi$ 
(the support of $\sigma$), there is a transformation
$H(y) \in \cG$ and a constant $c(y) > 0$ such that
\begin{equation} \label{4.1}
\lambda(y) := \lambda(H(y)) \in [\lambda_2^{-1},\lambda_1^{-1}]
\end{equation}
and
\begin{equation} \label{4.2}
H(y)_\sharp \sigma = c(y) \sigma.
\end{equation}
By definition of $\cG$ (see \eqref{1.6}),
the linear part of $H(y)$ is of the form $\lambda(y) R(y)$, where $R(y)$ is a linear
isometry. Since $H(y)$ fixes $y$, this means that
\begin{equation} \label{4.3}
H(y)(u) = y + \lambda(y) R(y)(u-y)
\ \text{ for } u\in \R^n.
\end{equation}
The next lemma allows us to replace $H(y)$ with one of its large powers,
chosen so that its isometric part is very close to the identity.

\begin{lemma} \label{t4.1}
For each choice of $\varepsilon > 0$, there is an integer $m_0$, 
that depends only on $\varepsilon$ and $n$, such that for each $y\in \Xi$
and each integer $\ell \geq 1$, there is an integer $m(y) \in [1,m_0]$ such that
\begin{equation} \label{4.4}
|| R(y)^{m(y)\ell} - I || \leq \varepsilon.
\end{equation}
\end{lemma}

\ms
\begin{proof}  
Here $I$ is the identity mapping. We use the compactness of the group of 
linear isometries of $\R^n$ to choose $m_0$ large enough so that if $R_1, \ldots, R_{m_0}$ 
are linear isometries, we can find integers $m_1, m_2$
such that $1 \leq m_1 < m_2 < m_0$ and $|| R_{m_2} - R_{m_1}|| \leq \varepsilon$.
We apply this with $R_m = R(y)^{m \ell}$ to find $m_1$ and $m_2$ such that
$|| R(y)^{m_2 \ell} - R(y)^{m_1 \ell}|| \leq \varepsilon$. Then 
$|| R(y)^{(m_2-m_1)\ell} - I|| \leq \varepsilon$, as needed.
\end{proof}

\ms
Next we study elementary properties of $\Xi$. Notice that by \eqref{4.2},
$H(y)(\Xi) = \Xi$, and iterations also yield
\begin{equation} \label{4.5}
H(y)^m(\Xi) = \Xi \ \text{ for } m \geq 1.
\end{equation}

\begin{lemma} \label{t4.2}
The set $\Xi$ is convex.
\end{lemma}

\ms
\begin{proof}  
Let $x, y \in \Xi$ be given. 
Our goal is to show that the segment $[x,y]$ is contained in $\Xi$. 
For each $\varepsilon > 0$ and $\ell \geq 1$, we construct a sequence $\{ y_k \}$ in $\Xi$ 
(depending on $\varepsilon$ and $\ell$) which will allow us to estimate how far $[x,y]$ is from $\Xi$.
We start with $y_0 = y$. 
If $k \geq 0$ and $y_k \in \Xi$ has been defined, we define $y_{k+1}$ as follows.
Set
\begin{equation} \label{4.6}
H_k = H(y_k)^{m(y_k) \ell} \ \text{ and } \ y_{k+1} = H_k(x).
\end{equation}
By \eqref{4.5} and since $x\in\Xi$, $y_{k+1} \in \Xi$.
Let us show that for $\ell$ large and $\varepsilon$ small, the $y_k$
stay close to the segment $[x,y]$ and converge (slowly) to $x$. First
observe that by iterations of \eqref{4.3}, $H_k$ is given by
\begin{equation} \label{4.7}
H_k(u) = y_k + \lambda'_k R_k(u-y_k)
\ \text{ for } u\in \R^n,
\end{equation}
with
\begin{equation} \label{4.8}
\lambda'_k = \lambda(H(y_k))^{m(y_k) \ell}
\in [\lambda_2^{-m_0 \ell},\lambda_1^{-\ell}]\hbox{  and  hence }\lambda'_k<1.
\end{equation}
In addition, $R_k = R(y_k)^{m(y_k) \ell}$ and therefore, by \eqref{4.4},
$||R_k -I || \leq \varepsilon$.

Set $r_k = |x-y_k|$ and 
$y_{k+1}^\ast = y_k + \lambda'_k (x-y_k)$. Notice that by \eqref{4.7} and \eqref{4.4},
\begin{eqnarray} \label{4.9}
|y_{k+1}-y_{k+1}^\ast| 
&=& \big|[y_k+\lambda'_k R_k (x-y_k)] - [y_k+\lambda'_k(x-y_k)] \big|
\nonumber\\
&=& \big|\lambda'_k [R_k - I](x-y_k)\big|
\leq \varepsilon \lambda'_k r_k.
\end{eqnarray}
Since
\begin{equation} \label{4.10}
|y_{k+1}^\ast-x| = (1-\lambda'_k) |y_{k}-x| = (1-\lambda'_k) r_k,
\end{equation}
we also get that if $\varepsilon < 1/2$,
\begin{eqnarray} \label{4.11}
r_{k+1} &=& |y_{k+1}-x| \leq |y_{k+1}^\ast-x|+|y_{k+1}-y_{k+1}^\ast|
\nonumber\\
&\leq& (1-\lambda'_k) r_k + \varepsilon \lambda'_k r_k
\leq (1-\lambda'_k/2) r_k.
\end{eqnarray}
Then since  $y_{k+1}^\ast \in [x,y_k]$, and by \eqref{4.9} and \eqref{4.11} we have
\begin{equation} \label{4.12}
\dist(y_{k+1},[x,y_k]) \leq |y_{k+1}-y_{k+1}^\ast|
\leq \varepsilon \lambda'_k r_k \leq 2 \varepsilon (r_k - r_{k+1}).
\end{equation}
By elementary geometry, 
\begin{equation} \label{4.13}
\dist(y_{k+1},[x,y]) \leq \dist(y_{k+1},[x,y_k])+\dist(y_{k},[x,y]).
\end{equation}
An iteration of \eqref{4.12} combined with \eqref{4.13} yields
\begin{equation} \label{4.14}
\dist(y_{k+1},[x,y]) \leq 2 \varepsilon \sum_{0 \leq j \leq k} (r_j - r_{j+1})
\leq 2 \varepsilon r_0 = 2 \varepsilon |x-y|.
\end{equation}
Notice that  by \eqref{4.11} and \eqref{4.8} $r_k$ tends to $0$. Thus we have constructed a sequence $\{ y_k \}$
in $\Xi$, which goes from $y=y_0$ to $x = \lim_{k \to \infty} y_k$.
The points $y_k$ lie within $2 \varepsilon |x-y|$ from $[x,y]$. 
Using the  definition of $y_{k+1}^\ast$, \eqref{4.9}, and \eqref{4.8} we can estimate 
their successive distances
\begin{eqnarray} \label{4.15}
|y_{k+1}-y_{k}| &\leq& |y_{k+1}-y_{k+1}^\ast| + |y_{k+1}^\ast-y_{k}| 
\leq \varepsilon \lambda'_k r_k + \lambda'_k r_k
\nonumber\\
&\leq& 2 \lambda'_k r_k \leq 2 \lambda_1^{-\ell} |x-y|.
\end{eqnarray}
Let $z_k$ be the orthogonal projection of $y_k$ into the line containing $x$ and $y$. 
Note that by \eqref{4.14} 
$z_k\in \big[x-2\varepsilon|x-y|\frac{y-x}{|x-y|}, y+2\varepsilon|x-y|\frac{y-x}{|x-y|}\big]$. 
By \eqref{4.15}, $|z_{k+1}-z_k|\le  2 \lambda_1^{-\ell} |x-y|$ and by the 
definition of $z_k$, $|z_k-y_k|\le 2 \varepsilon |x-y|$. Therefore 
every point of $[x,y]$ lies within
$2 (\varepsilon + \lambda_1^{-\ell}) |x-y|$ of some $y_k$, that is each  point of $[x,y]$ is at most $2 (\varepsilon + \lambda_1^{-\ell}) |x-y|$  away from $\Xi$.
Choosing
$\varepsilon$ arbitrarly small and $\ell$ arbitrarly large, we get
that $[x,y] \subset \Xi$. Lemma~\ref{4.2} follows.
\end{proof}

\ms
\begin{lemma} \label{t4.3}
The set $\Xi$ is a vector subspace of $\R^n$.
\end{lemma}

\begin{proof}  
Let $V$ be the smallest affine subspace of $\R^n$ that contains
$\Xi$, and let $d$ denote its dimension. Choose $d+1$ affinely
independent points $y_0, \ldots y_{d}$ in $V$
(this means that the vectors $y_j-y_0$, $j \geq 1$, are linearly independent),
and set $y = {1 \over d+1} \sum_{j=0}^d y_j$. 
By Lemma~\ref{t4.2}, $\Xi$ contains the convex hull of the $y_j$,
so there is a small radius $r > 0$ such that $V \cap B(y,r) \subset \Xi$.

Recall from \eqref{4.5} that $H(y)^m(\Xi) = \Xi$ for $m \geq 1$.
Applying the bijection $H(y)^{-m}$ to both sides, we see that
$H(y)^{-m}(\Xi) = \Xi$ for $m \geq 1$, so 
\begin{equation} \label{4.16}
H(y)^{-m}(V \cap B(y,r)) \subset \Xi.
\end{equation}
We know that $H(y)^{-m}(V \cap B(y,r))$ is a nontrivial open subset
of the affine space $H(y)^{-m}(V)$, and since $\Xi \subset V$ we get that
$H(y)^{-m}(V) \subset V$. Then by a dimension count $H(y)^{-m}(V) = V$,
and by the description \eqref{4.3} of $H(y)$, we see that
\begin{equation} \label{4.17}
H(y)^{-m}(V \cap B(y,r)) = V \cap B(y, \lambda(H(y))^{-m} r).
\end{equation}
Recall that $\lambda(H(y)) < 1$; then $\lambda(H(y))^{-m} r$ is as large as we want by picking $m$ as large as we need.
Hence \eqref{4.16} guarantees that $\Xi$ contains $V$, and this means that $\Xi = V$.

Note that Remark 14.4 (2) in \cite{Mattila} ensures that $0\in\Xi$, 
as $\Xi$ is the support of $\sigma$ and $\sigma\in\Tan(\mu,x_0)$ where $\mu$ is a doubling 
measure and $x_0\in \Sigma$. Hence $\Xi = V$ is a vector space.
This completes our proof of Lemma \ref{t4.3}.
\end{proof}

\ms
Now we study the distribution of $\sigma$ on $\Xi$.
If $\Xi$ is reduced to the origin, then $\sigma$ is a Dirac mass, and
Dirac masses lie in $\cF_0$. Thus in this case $\sigma$ is trivially flat.
We may now assume that $\Xi$ is a vector
space of dimension $d > 0$.

\begin{lemma} \label{t4.4}
There is a dimension $D \geq 0$ such that $\sigma$ is Ahlfors-regular of dimension $D$,
which means that
\begin{equation} \label{4.19}
C^{-1} \rho^D \leq \sigma(B(y,\rho)) \leq C\rho^D
\ \text{ for $y\in \Xi$ and $\rho > 0$}
\end{equation}
for some constant $C \geq 1$.
\end{lemma}

\ms
\begin{proof} 
Set $\lambda(y) = \lambda(H(y))$, and 
recall from \eqref{4.1} that $1 < \lambda_1 \leq \lambda(y)^{-1} \leq \lambda_2$.
Notice that by \eqref{4.2}, for $y\in \Xi$ and $r > 0$
\begin{eqnarray} \label{4.20}
\sigma(B(y,\lambda(y)^{-1} r)) &=& 
\sigma(H(y)^{-1}(B(y,r))) 
= H(y)_\sharp \sigma(B(y,r)) 
\nonumber
\\
&=&  c(y) \sigma(B(y,r)).
\end{eqnarray}
Iterating we obtain for $m\ge 0$
\begin{equation} \label{4.21}
\sigma(B(y,\lambda(y)^{-m} r)) = c(y)^m \sigma(B(y,r)).
\end{equation}
Applying \eqref{4.21} to $\lambda(y)^{m} r$ instead of $r$
we have that \eqref{4.21} also holds for $m\le 0$.
Observe that since $\lambda(y) < 1$ and $\sigma(B(y,r)) > 0$ when $y\in \Xi$,
\eqref{4.21} yields $c(y) \geq 1$.
If $c(y) = 1$, then $\sigma(B(y,\lambda(y)^{-m} r))=\sigma(B(y,r))$
for all $m\in \bZ$, and $\sigma$ is a Dirac mass.
This case was excluded before the statement of the lemma, so $c(y) > 1$. 

Now let $\rho >0$ be given, and choose $m$ such that 
\begin{equation} \label{4.22}
\lambda(y)^{-m} \leq \rho \leq \lambda(y)^{-m-1}.
\end{equation}
By \eqref{4.21} applied to $r=1$,
\begin{equation} \label{4.23}
c(y)^m \sigma(B(y,1)) \leq \sigma(B(y,\rho)) \leq c(y)^{m+1} \sigma(B(y,1)),
\end{equation}
hence, letting $\ell = \log(\sigma(B(y,1)))$, \eqref{4.23} yields
\begin{equation}\label{4.23A} 
m \log(c(y)) + \ell \leq \log(\sigma(B(y,\rho)))  
\leq (m+1) \log(c(y)) + \ell.
\end{equation}
By \eqref{4.22}
\begin{equation}\label{4.22A}
m \log(\lambda(y)^{-1}) \leq \log\rho \leq (m+1) \log(\lambda(y)^{-1}).
\end{equation}
Hence combining \eqref{4.23A} and \eqref{4.22A} we have
\begin{equation} \label{4.24}
\lim_{\rho \to +\infty} {\log(\sigma(B(y,\rho))) \over \log\rho}
= {\log(c(y)) \over \log(\lambda(y)^{-1})}=: D(y).
\end{equation}
We
claim that $D(y)$ does not depend on $y$.
Indeed, if $z \in \Xi$, observe that $B(y,\rho) \subset B(z,\rho+|z-y|)$, hence
\begin{eqnarray} \label{4.25}
D(y) &=& \lim_{\rho \to +\infty} {\log(\sigma(B(y,\rho))) \over \log\rho}
\leq \liminf_{\rho \to +\infty} {\log(\sigma(B(z,\rho+|z-y|))) \over \log\rho}
\nonumber
\\
&=&  \liminf_{\rho \to +\infty} {\log(\sigma(B(z,\rho+|z-y|))) \over \log(\rho+|z-y|)}
= D(z);
\end{eqnarray}
the opposite inequality also holds exchanging the roles of $y$ and $z$.
Let $D$ be the common value of the $D(y)$ for $y\in \Xi$. 
By definition (see \eqref{4.24}), $\lambda(y)^{-D} = c(y)$, and  using \eqref{4.22}
we can rewrite \eqref{4.23} as
\begin{equation} \label{4.26}
\lambda(y)^{-m D} \sigma(B(y,1)) \leq \sigma(B(y,\rho)) 
\leq \lambda(y)^{-D} \lambda(y)^{-m D} \sigma(B(y,1)).
\end{equation}
Thus, by \eqref{4.22},
\begin{equation} \label{4.27}
\lambda(y)^{D} \rho^D \sigma(B(y,1)) \leq \sigma(B(y,\rho)) 
\leq \lambda(y)^{-D} \rho^{D} \sigma(B(y,1)),
\end{equation}
which yields \eqref{4.19} with 
$C = \lambda_2^D \max(\sigma(B(y,1)), \sigma(B(y,1))^{-1})$ 
(because \eqref{4.1} guarantees that $\lambda(y)^{-1} \leq \lambda_2$).
\end{proof}

\ms
\begin{lemma} \label{t4.5}
Let $d$ be the dimension of the vector space $\Xi$, and denote by
$\nu=\H^d\res\Xi$ the restriction of $\H^d$ to $\Xi$. Let $D$ be as in \eqref{4.19}.
Then $d=D$, and there exists a constant $c_0>0$ such that $\sigma=c_0 \nu$.
\end{lemma}

\ms
\begin{proof}  
Since $\sigma$ is Ahlfors regular of dimension $D$ (by Lemma \ref{t4.4}),
a standard covering argument (see for instance Lemma 18.11 in \cite{MSbook})
guarantees that there exists a constant $C >0$ such that
\begin{equation} \label{4.28}
C^{-1} \H^D\res{\Xi} \leq \sigma \leq C \H^D\res{\Xi},
\end{equation}
and $\Xi$ is a $D$-dimensional Ahlfors regular set. Hence $D=d$.

By \eqref{4.28}, $\sigma$ is absolutely continuous with respect to $\nu$, 
and the Radon-Nikodym derivative of $\sigma$ with respect to $\nu$ is bounded.
Thus there is a bounded function $f$ on $\Xi$ such that $\sigma = f \nu$.

We now show that $f$ is constant. First observe that since 
$d = D = {\log(c(y)) \over \log(\lambda(y)^{-1})}$ (by \eqref{4.24}),
\eqref{4.21} guarantees that for $y\in \Xi$, $r > 0$, and $m \in \bZ$
\begin{equation} \label{4.29}
\sigma(B(y,\lambda(y)^{-m} r)) = c(y)^m \sigma(B(y,r))
= \lambda(y)^{- md} \sigma(B(y,r)).
\end{equation}
 Since 
$\nu(B(y,\lambda(y)^{-m} r)) = \lambda(y)^{- md} \nu(B(y,r))$,
we may rewrite \eqref{4.29} as
\begin{equation} \label{4.30}
{\sigma(B(y,\lambda(y)^{-m} r)) \over \nu(B(y,\lambda(y)^{-m} r))}
= {\sigma(B(y, r)) \over \nu(B(y, r))}.
\end{equation}

The Lebesgue differentiation theorem says that for $\nu$-almost every $y\in \Xi$,
\begin{equation} \label{4.31}
f(y) = \lim_{\rho \to 0} {\sigma(B(y,\rho)) \over \nu(B(y,\rho))}.
\end{equation}
For such an $y$ and every $r > 0$, by \eqref{4.31} and \eqref{4.30} for $-m$
\begin{equation} \label{4.32}
f(y) = \lim_{m \to \infty} 
{\sigma(B(y,\lambda(y)^{m} r)) \over \nu(B(y,\lambda(y)^{m}r))}
= {\sigma(B(y,r)) \over \nu(B(y,r))}.
\end{equation}
That is, 
\begin{equation} \label{4.33}
\sigma(B(y,r)) = f(y) \nu(B(y,r)) \ \text{ for } r > 0.
\end{equation}
If $z \in \Xi$ is another Lebesgue point of $f$, since  $B(z,r) \subset B(y,r+|y-z|)$ we have
\begin{eqnarray}  \label{4.34}
f(z) &=& \lim_{r \to \infty} {\sigma(B(z,r)) \over \nu(B(z,r))}
\leq \liminf_{r \to \infty} {\sigma(B(y,r+|y-z|)) \over \nu(B(z,r))}
\nonumber \\
&=& \liminf_{r \to \infty} {\sigma(B(y,r+|y-z|)) \over \nu(B(y,r+|y-z|))}
= f(y).
\end{eqnarray}
Similarly $f(y) \leq f(z)$, and $f$ is constant.
\end{proof}

This completes the proof of Theorem~\ref{t1.3}. 
In fact we have proved that if $\sigma\in \Tan(\mu, x_0)$ with $x_0\in\Sigma_0$ 
(see \eqref{1.24}) then $\sigma=c_0\H^d\res\Xi$ for some vector space $\Xi$.

\section{A smoother version of the Wasserstein $\W_1$ distance}
\label{S5}

So far we managed to work with the distance $\W_1$ defined by \eqref{1.5},
but for the proof of Theorem~\ref{t1.2}, it is more convenient to use
a slightly smoother variant, which attenuates the possible discontinuities in $r>0$ 
of the 
normalizing factors $\mu^{x,r}(B(0,1))^{-1} = \mu(B(x,r))^{-1}$.

Let us choose a smooth radial function $\varphi$ such that
\begin{equation} \label{5.1}
\1_{B(0,1/2)} \leq \varphi \leq \1_{B(0,1)};
\end{equation}
If $\mu$ and $\nu$ are two Radon measures such that
\begin{equation} \label{5.2}
\mu(B(0,1/2)) > 0 \ \text{ and } \ \nu(B(0,1/2)) > 0,
\end{equation}
we define a new distance $\W_{\varphi}(\mu,\nu)$ by
\begin{equation} \label{5.3}
\W_{\varphi}(\mu,\nu) = \sup_{\psi \in \Lip_1(\bB)}
\av{{\int\psi \varphi d\mu \over \int \varphi d\mu}
-{\int\psi \varphi d\nu \over \int \varphi d\nu}}.
\end{equation}
Recall that $\Lip_1(\bB)$ is defined near \eqref{1.3}. 
The distance $\W_1$ above essentially corresponds to $\varphi = \1_{\bB}$ here.

We required \eqref{5.2} (and \eqref{5.1}) to make sure that we do not divide
by $0$. But notice that even when $\mu(B(0,1/2))$ or $\nu(B(0,1/2))$
is very small, we always get that 
\begin{equation} \label{5.4}
\W_{\varphi}(\mu,\nu) \leq 2,
\end{equation}
because $|\int\psi \varphi d\mu| \leq \int \varphi d\mu$ and similarly
for $\nu$. Note that
\begin{equation} \label{5.5}
\W_{\varphi}(a\mu,b\nu) = \W_{\varphi}(\mu,\nu) \ \text{ for } a, b > 0,
\end{equation}
so we do not need to normalize $\mu$ and $\nu$ in advance.
Finally observe that $\W_{\varphi}$ satisfies the triangle inequality.
That is, if $\sigma$ is a third Radon measure such that $\sigma(B(0,1/2)) > 0$,
if follows at once from the definition that
\begin{equation} \label{5.6}
\W_{\varphi}(\mu,\sigma) \leq \W_{\varphi}(\mu,\nu) + \W_{\varphi}(\nu,\sigma).
\end{equation}

Let us check that if we restrict to measures that are not to small on $B(0,1/2)$, 
then $\W_1$ controls $\W_{\varphi}$.

\begin{lemma} \label{t5.1}
Let $\mu$ and $\nu$ be Radon measures such that \eqref{5.2} holds and
\begin{equation} \label{5.7}
\mu(\bB) = \nu(\bB) =1.
\end{equation}
Then
\begin{equation} \label{5.8}
\W_{\varphi}(\mu,\nu) 
\leq {1+ 2 ||\varphi||_{lip} \over \mu(B(0,1/2))} \, \W_1(\mu,\nu),
\end{equation}
where $||\varphi||_{lip}$ denotes the Lipschitz norm of $\varphi$.
\end{lemma}

\ms

The fact that the estimate is not symmetric is not an issue. 
In particular we shall apply \eqref{5.8} to doubling measures $\mu$ and $\nu$; in this 
case $\mu(B(0,1/2))\sim\mu(\bB)=1=\nu(\bB)\sim\nu(B(0,1/2))$.

\begin{proof}  
Let $\psi \in \Lip_1(\bB)$ be given. The definition \eqref{1.5}, applied to
$\psi \varphi$, yields
\begin{equation} \label{5.9}
\Big|\int \psi \varphi d\mu - \int \psi \varphi d\nu \Big|
\leq ||\psi \varphi||_{lip} \W_1(\mu,\nu)
\leq (1+||\varphi||_{lip}) \W_1(\mu,\nu).
\end{equation}
The same definition, applied to $\varphi$ itself, yields
\begin{equation} \label{5.10}
\Big|\int \varphi d\mu - \int \varphi d\nu \Big|
\leq ||\varphi||_{lip} \W_1(\mu,\nu).
\end{equation}
Set
\begin{equation} \label{5.11}
\Delta = \av{{\int\psi \varphi d\mu \over \int \varphi d\mu}
-{\int\psi \varphi d\nu \over \int \varphi d\nu}}
\end{equation}
and write
\begin{equation} \label{5.12}
\Delta = \av{ {a \over b} - {c \over d}} = \av{{ad-bc \over bd}}
= {| d(a-c)+c(d-b)| \over bd}
\end{equation}
with  $a = \int\psi \varphi d\mu$, $b= \int \varphi d\mu$, 
$c= \int\psi \varphi d\nu$, and $d=\int \varphi d\nu$.
Notice that $|c| \leq d$ because $\psi \in \Lip_1(\bB)$ and $\varphi \geq 0$.
Also, $b \geq \mu(B(0,1/2))$ by \eqref{5.1}. Hence by \eqref{5.9} and \eqref{5.10}; 
 we have
\begin{equation} \label{5.13}
\Delta \leq  {|a-c| + |d-b| \over b}
\leq {1+ 2 ||\varphi||_{lip} \over \mu(B(0,1/2))} \, \W_1(\mu,\nu).
\end{equation}
Taking the supremum over $\psi \in \Lip_1(\bB)$ in \eqref{5.13} ( recall \eqref{5.11})  yields \eqref{5.8}.
\end{proof}

The following lemma specifies the sense in which $\W_{\varphi}$ is more stable that $\W_1$.

\begin{lemma} \label{t5.2}
Let $\mu$ and $\nu$ be Radon measures and let $\theta\in (0, 1/2]$ be such that
\begin{equation} \label{5.14}
\mu(B(0,\theta/2)) > 0 \ \text{ and } \ \nu(B(0,\theta/2)) > 0,
\end{equation}
and define new measures $\mu_1$ and $\nu_1$ by
\begin{equation} \label{5.15}
\mu_1(A) = \mu(\theta A) \text{  and  } \nu_1(A) = \nu(\theta A) 
\ \text{ for } A \subset \R^n.
\end{equation}
Then
\begin{equation} \label{5.16}
\W_{\varphi}(\mu_1,\nu_1) 
\leq \theta^{-1}(1+4||\varphi||_{lip}) \, {\mu(B(0,1)) \over \mu(B(0,\theta/2))} \,
\W_{\varphi}(\mu,\nu).
\end{equation}
\end{lemma} 

\ms
As in \eqref{5.8} the estimate is not symmetric in $\mu$ and $\nu$, but is nonetheless valid.
We require that $\mu(B(0,\theta/2)) \neq 0$ and $\nu(B(0,\theta/2)) \neq 0$ to make sure that 
$\W_{\varphi}(\mu_1,\nu_1)$ is easily defined. Often $\mu$ is the restriction to $\bB$ of a doubling
measure and its support contains the origin; then ${\mu(B(0,1)) \over \mu(B(0,\theta/2))} \geq C^{-1}$,
for some $C$ that depends on $\theta$ and the doubling constant $C_\delta$.

\begin{proof}  
Let $\psi \in \Lip_1(\bB)$ be given; we want to control the quantity
\begin{equation} \label{5.17}
\Delta = \av{{\int\psi \varphi d\mu_1 \over \int \varphi d\mu_1}
-{\int\psi \varphi d\nu_1 \over \int \varphi d\nu_1}}
= : \av{ {a \over b} - {c \over d}} 
\end{equation}
(as above, but with integrals relative to $\mu_1$ and $\nu_1$).
Notice that by \eqref{5.15},
\begin{equation} \label{5.18}
a = \int\psi \varphi d\mu_1 = \int\psi(\theta^{-1} x) \varphi(\theta^{-1} x) d\mu(x)
= \int\psi(\theta^{-1} x) \varphi(\theta^{-1} x) \varphi(x)^2 d\mu(x),
\end{equation}
where we just use the fact that $\varphi(\theta^{-1} x)= 0$ when $x\in \R^n \sm B(0,1/2)$,
and the special shape of $\varphi$ in \eqref{5.1}, to add an extra $\varphi^2(x)$.
Similarly,
\begin{equation} \label{5.19}
c = \int\psi \varphi d\nu_1 = \int\psi(\theta^{-1} x) \varphi(\theta^{-1} x) \varphi^2(x) d\nu(x).
\end{equation}
The same computations without $\psi$ yield
\begin{equation} \label{5.20}
b = \int \varphi d\mu_1 = \int \varphi(\theta^{-1} x) \varphi(x)^2 d\mu(x),
\end{equation}
\begin{equation} \label{5.21}
d = \int \varphi d\nu_1 = \int \varphi(\theta^{-1}x) \varphi(x)^2 d\nu(x).
\end{equation}
It is also useful to introduce
\begin{equation} \label{5.22}
e = \int \varphi d\mu \ \text{ and } \ f = \int \varphi d\nu.
\end{equation}
Let us first estimate $\delta_1 = {a \over e} - {c \over f}$.
We want to apply the definition of $\W_{\varphi}(\mu,\nu)$ to the function $\Psi$
defined by $\Psi(x) = \psi(\theta^{-1}x) \varphi(\theta^{-1} x) \varphi(x)$. 
Notice that $\Psi$ is supported in $\bB$ (this is why we added $\varphi(x)$), and its Lipschitz norm 
is at most $\theta^{-1}(1+||\varphi||_{lip}) + ||\varphi||_{lip} \leq \theta^{-1}(1+2||\varphi||_{lip})$. 
Thus \eqref{5.3} yields
\begin{equation} \label{5.23}
|\delta_1|= \av{ {a \over e} - {c \over f}} \leq \theta^{-1}||\Psi ||_{lip} \W_{\varphi}(\mu,\nu)
\leq  \theta^{-1}(1+2||\varphi||_{lip}) \W_{\varphi}(\mu,\nu).
\end{equation}
We can also apply the definition of $\W_{\varphi}(\mu,\nu)$ to $\varphi(\theta^{-1} x) \varphi(x)$, 
whose Lipschitz norm is at most $2\theta^{-1} ||\varphi||_{lip}$,
and we get that
\begin{equation} \label{5.24}
|\delta_2| :=\av{ {b \over e} - {d \over f}} \leq 2\theta^{-1} ||\varphi||_{lip} \W_{\varphi}(\mu,\nu).
\end{equation}
Thus 
$$
{a \over b} = {e \over b} {a\over e}= {e \over b}\Big({c \over f} + \delta_1\Big)
= {e \over b}{c \over d} {d \over f} 
+ {e \delta_1\over b}
= {e \over b}{c \over d}\Big({b \over e}+\delta_2 \Big)+ {e \delta_1\over b}
= {c \over d} + {e c \delta_2 \over b d} + {e \delta_1\over b},
$$
where $\delta_1$ and $\delta_2$ are as in\eqref{5.23} and \eqref{5.24}.
Thus
\begin{equation} \label{5.25}
\Delta =  \av{ {a \over b} - {c \over d}} 
\leq {|e c \delta_2| \over b d} + {|e \delta_1|\over b}.
\end{equation}
Now $|c| \leq d$ because $\varphi \geq 0$ and $|\psi| \leq 1$,
$e = \int \varphi d\mu \leq \mu(B(0,1))$, and 
$b = \int \varphi(\theta^{-1} x) \varphi(x)^2 d\mu(x) \geq \mu(B(0,\theta/2))$
because of \eqref{5.1}. Thus we have
\begin{equation} \label{5.26}
\Delta \leq (\delta_1 + \delta_2) {e \over b}
\leq\theta^{-1} (1+4||\varphi||_{lip}) \W_{\varphi}(\mu,\nu)
\, {\mu(B(0,1)) \over \mu(B(0,\theta/2))}.
\end{equation}
Noting \eqref{5.17} and taking the supremum over $\psi \in \Lip_1(\bB)$,
 we obtain \eqref{5.16}.
\end{proof}

The next lemma is used in Section \ref{S8}. It shows that the distance
function $\W_{\varphi}$ also controls $\W_1$ in some averaged way.
Thus $\W_1$ and $\W_{\varphi}$ are basically interchangeable.

\begin{lemma} \label{t5.3}
Let $\mu$ and $\nu$ are Radon measures such that
$\mu(B(0,1/4)) > 0$ and $\nu(B(0,1/4))>0$.
Define $\mu_t$ and $\nu_t \,$, $1/4 \leq t \leq 1/2$, by
\begin{equation} \label{5.27}
\mu_t(A) = {\mu(tA) \over \mu(B(0,t))}
\ \text{ and } \ 
\nu_t(A) = {\mu(tA) \over \nu(B(0,t))}
\ \text{ for } A \subset \R^n.
\end{equation}
Then
\begin{equation} \label{5.28}
\fint_{1/4}^{1/2} \W_1(\mu_t,\nu_t) dt
\leq {(8+||\varphi||_{lip})\mu(\bB) \over \mu(B(0,1/4))} \, \W_{\varphi}(\mu,\nu). 
\end{equation}
\end{lemma}

\ms
\begin{proof} 
First notice that the statement does not change when we multiply
$\mu$ and $\nu$ by positive constants. So we may assume that
\begin{equation} \label{5.29}
\int \varphi d\mu = \int \varphi d\nu = 1.
\end{equation}
Next fix $t \in [1/4,1/2]$ and $\psi \in \Lip_1(\bB)$. We want to estimate
\begin{equation} \label{5.30}
\int \psi d\mu_t - \int \psi d\nu_t
= {\int \psi(t^{-1}x) d\mu(x) \over \mu(B(0,t))} - {\int \psi(t^{-1}x) d\nu(x) \over \nu(B(0,t))}
= : {a \over b} -{c \over d}.
\end{equation}
As before, 
\begin{equation} \label{5.31}
\Delta = {|ad-bc| \over bd} = {|d(a-c) - c(b-d)| \over bd}
\leq {|a-c| + |b-d| \over b}
\end{equation}
because $|c| = |\int \psi(t^{-1}x) d\nu(x)| \leq ||\psi||_\infty \nu(B(0,t)) \le d$.
Notice that $\varphi(x) = 1$ when $\psi(t^{-1}x) \neq 0$ since this implies that 
$|t^{-1}x|\le 1$ and hence $|x|\le t\le 1/2$. Thus
\begin{eqnarray} \label{5.32}
|a-c| &=& \av{\int \psi(t^{-1}x) d\mu(x)-\int \psi(t^{-1}x) d\nu(x)}
\\
&=& \av{\int \psi(t^{-1}x) \varphi^2(x) d\mu(x)-\int \psi(t^{-1}x) \varphi^2(x) d\nu(x)}.
\nonumber 
\end{eqnarray}
We apply the definition \eqref{5.3} of $\W_\varphi$ with the function $x \to \psi(t^{-1}x) \varphi(x)$,
which is supported in $\bB$ and $(t^{-1}+||\varphi||_{lip})$-Lipschitz.
We obtain that
\begin{equation} \label{5.33}
|a-c| \leq (t^{-1}+||\varphi||_{lip})\W_\varphi(\mu,\nu)
\end{equation}
Notice also that
\begin{equation} \label{5.34}
{1 \over b} = {\int \varphi d\mu \over \mu(B(0,t))}
\leq {\mu(\bB) \over \mu(B(0,t))} \leq {\mu(\bB) \over \mu(B(0,1/4))}.
\end{equation}
Thus
\begin{eqnarray}  \label{5.35}
\Delta &\leq&  {|a-c| + |b-d| \over b}
\leq {\mu(\bB) \over \mu(B(0,1/4))} \, [(t^{-1}+||\varphi||_{lip})\W_\varphi(\mu,\nu) + |b-d|]
\hskip0.8cm\,
\nonumber \\
&=& {\mu(\bB) \over \mu(B(0,1/4))} \, 
\big[(t^{-1}+||\varphi||_{lip})\W_\varphi(\mu,\nu) + |\mu(B(0,t)) - \nu(B(0,t))| \big].
\end{eqnarray}
We take the supremum over $\psi \in \Lip_1(\bB)$ and get that
\begin{equation} \label{5.36}
\W_1(\mu_t,\nu_t) \leq {\mu(\bB) \over \mu(B(0,1/4))} \, 
\big[(t^{-1}+||\varphi||_{lip})\W_\varphi(\mu,\nu) + |\mu(B(0,t)) - \nu(B(0,t))| \big].
\end{equation}
Since $t^{-1}\in[2,4]$, \eqref{5.28} will follow as soon as we prove that
\begin{equation} \label{5.37}
\int_{[1/4,1/2]} |\mu(B(0,t))-\nu(B(0,t))| dt \leq \W_\varphi(\mu,\nu).
\end{equation}

Let $h$ be a bounded measurable function, defined on $[1/4,1/2]$;
we want to evaluate
\begin{equation} \label{5.38}
I_h = \int_{[1/4,1/2]} h(t) [\mu(B(0,t))-\nu(B(0,t))] dt.
\end{equation}
Observe that by Fubini
$$
\int_{[1/4,1/2]} h(t) \mu(B(0,t)) dt
= \int_{x\in B(0,1/2)} \Big\{\int \1_{t\in [1/4,1/2]} \1_{t > |x|} h(t) dt \Big\} d\mu(x),
$$
and similarly for $\nu$.
Set $\psi_h(x) = \int \1_{t\in [1/4,1/2]} \1_{t > |x|} h(t) dt$.
This is a $||h||_\infty$-Lipschitz function of $|x|$, which vanishes when 
$|x| \geq 1/2$, so by \eqref{5.1}, \eqref{5.3} and the normalization in \eqref{5.29}
we have
\begin{eqnarray}  \label{5.39}
| I_h | &=& \av{\int_{B(0,1/2)} \psi_h d\mu - \int_{B(0,1/2)} \psi_h d\nu }
\\
&=& \av{ \int \psi_h \varphi d\mu - \int \psi_h \varphi d\nu }
\leq ||h||_\infty \W_\varphi(\mu,\nu).
\nonumber 
\end{eqnarray}
Thus
\eqref{5.39} holds for all bounded measurable functions $h$ defined on $[1/4,1/2]$,
and \eqref{5.37} follows by duality. We saw earlier that \eqref{5.37}
implies \eqref{5.28}. Lemma \ref{t5.3} follows.
\end{proof}

We conclude this section with an easy observation concerning the behavior of 
$\W_\varphi(\mu,\nu)$ when taking weak limits.

\ms
\begin{lemma} \label{t5.4}
Let $\mu$ and $\nu$ satisfy \eqref{5.2}, suppose that
$\mu$ is the weak limit of some sequence $\{\mu_k \}$, and that
$\nu$ is the weak limit of some sequence $\{\nu_k \}$. Then
\begin{equation} \label{5.40}
\W_\varphi(\mu,\nu) \leq \liminf_{k \to \infty} \W_\varphi(\mu_k,\nu_k).
\end{equation}
\end{lemma}

\ms
\begin{proof}  
Set $L_k = \W_\varphi(\mu_k,\nu_k)$ and $L = \liminf_{k \to \infty} L_k$.
Notice that $\int\varphi d\mu = \lim_{k \to \infty} \int \varphi d\mu_k$
by weak convergence;
by \eqref{5.2}, this implies that $\int \varphi d\mu_k > 0$ for $k$ large.
The same argument applied to a continuous function $f \leq \1_{B(0,1/2)}$ 
such that $\int f d\mu > 0$
shows that $\mu_k(B(0,1/2)) > 0$ for $k$ large. Similar observations hold for $\nu$ and $\nu_k$.
For each $\psi \in \Lip_1(\bB)$, the weak convergence yields
$\int \psi\varphi d\mu = \lim_{k \to \infty} \int \psi\varphi d\mu_k$.
For $k$ large, 
$$
\av{{\int\psi \varphi d\mu_k \over \int \varphi d\mu_k}
-{\int\psi \varphi d\nu_k \over \int \varphi d\nu_k}} \leq L_k.
$$
Since each term has a limit and the denominators are bounded away from $0$, taking a $\liminf$ we have that
$$
\av{{\int\psi \varphi d\mu \over \int \varphi d\mu}
-{\int\psi \varphi d\nu \over \int \varphi d\nu}} \leq L.
$$
Taking the supremum over $\psi\in \Lip_1(\bB)$ we conclude that \eqref{5.40} holds.
\end{proof}

\section{ Uniqueness of the tangent measure at ``good points"}
\label{S6}
In this section show that for $x\in \Sigma_0 \cap \Sigma_1$,
$\Tan(\mu,x)$ is a one-dimensional set of flat measures.
Recall that $\Sigma_1$ and $\Sigma_0$ were defined in \eqref{1.16} and \eqref{1.24}
respectively.

\begin{lemma} \label{t6.1}
Let $\mu$ be a doubling measure, and let $x\in \Sigma_0 \cap \Sigma_1$.
Then there is a nonzero flat measure $\sigma$ such that 
$\Tan(\mu,x) = \big\{ c\sigma \, ; \, c > 0 \big\}$.
\end{lemma}

\ms
\begin{proof} 
Fix $\mu$ and $x\in \Sigma_0 \cap \Sigma_1$; without loss of generality,
we may assume that $x=0$. By Theorem \ref{t1.3}, we know that
\begin{equation} \label{6.1}
\Tan(\mu,0) \subset \cF
\end{equation}
where $\cF$ denotes the set of flat measures (see \eqref{1.21}).
By definition
of $\Sigma_1$,
\begin{equation} \label{6.2}
\int_0^1 \alpha_{\cD}(0,r) {dr \over r } < \infty.
\end{equation}

Thus it only remains to show that $\Tan(\mu,0)$ is one-dimensional. 
Our initial goal is to bound the $\W_{\varphi}$ distance for two different scaled dilations of $\mu$ 
by $ \alpha_{\cD}(0,\cdot)$ at the right scale.
For each $r \in (0,1/4)$ we use Chebyshev's inequality to find
$r_+ \in [2r, 4r]$ such that
\begin{equation} \label{6.3}
\alpha_{\cD}(0,r_+) 
\leq (2r)^{-1} \int_{2r}^{4r} \alpha_{\cD}(0,t) dt
\leq 2\int_{2r}^{4r} \alpha_{\cD}(0,t) {dt\over t}.
\end{equation}
By the definition of $\alpha_{\cD}(0,r_+)$ (see  \eqref{1.15}):
there is a transformation $G \in \cD(0,r_+)$, such that 
\begin{equation} \label{6.4}
\W_1(\mu_0^{G},\mu_0^{0,r_+}) \leq 2 \alpha_{\cD}(0,r_+).
\end{equation}
By the definition \eqref{1.14} of $\cD(0,r_+)$, $G$ is simply given by
$G(u) = \lambda u$, with $\lambda^{-1} = 
 \lambda(G)^{-1} \in [\lambda_1r_+,\lambda_2 r_+]$. Set
\begin{equation} \label{6.5}
r^\ast = \lambda^{-1} \in [\lambda_1r_+,\lambda_2 r_+];
\end{equation}
notice that $G$ is the homotety that sends $B(0,r^\ast)$ to
$\bB$, so $\mu_0^{G} = \mu_0^{0,r^\ast}$ 
(see near \eqref{1.11}) and now \eqref{6.4} says that
\begin{equation} \label{6.6}
\W_1(\mu_0^{0,r^\ast},\mu_0^{0,r_+}) \leq 2 \alpha_{\cD}(0,r_+).
\end{equation}

First apply Lemma \ref{t5.1} to the measures 
$\mu_0^{0,r^\ast}$ and $\mu_0^{0,r_+}$;\eqref{6.6} yields that
\begin{equation} \label{6.9}
\W_\varphi(\mu_0^{0,r^\ast},\mu_0^{0,r_+}) \leq C \alpha_{\cD}(0,r_+)
\end{equation}
(where we do not record any more the dependence on $\varphi$ or $C_\delta$).
Then we apply Lemma \ref{t5.2} to the same measures, with
$\theta = r/r^\ast$. Notice that $\theta \leq 1/2$ because $r^\ast \geq r_+ \geq 2r$.

Recall from \eqref{5.15} that $\mu_1$ is defined by
\begin{eqnarray}\label{6.10A}
\mu_1(A) &=&  \mu_0^{0,r^\ast}(\theta A) 
= \mu_0^{0,r^\ast}(r A/r^\ast) = \frac{\mu(r^\ast r A/r^\ast)}{\mu(B(0,r^\ast))}
\nonumber\\ 
&=& \frac{\mu(rA)}{\mu(B(0,r^\ast))}
= \frac{\mu(B(0,r))}{\mu(B(0,r^\ast))} \, \mu^{0,r}_0(A)
\end{eqnarray} 
by \eqref{1.11}. Similarly, $\nu_1$ is defined by
\begin{eqnarray}\label{6.10B}
\nu_1(A) &=&  \mu_0^{0,r_+}(\theta A) = \mu_0^{0,r_+}(r A/r^\ast)
= \frac{\mu(r_+ r A/r^\ast)}{\mu(B(0,r_+))}
\nonumber\\
&=& \frac{\mu(\rho(r)A)}{\mu(B(0,r_+))}
= \frac{\mu(B(0,\rho(r)))}{\mu(B(0,r_+))} \, \mu^{0,\rho(r)}_0(A)
\end{eqnarray} 
where
\begin{equation} \label{6.7}
\rho(r) =  { r  r_+  \over r^\ast}
\in [\lambda_2^{-1}r,\lambda_1^{-1} r].
\end{equation}

By \eqref{5.5}, \eqref{6.10A}, \eqref{6.10B}, Lemma \ref{t5.2}, 
the fact that $\mu_0^{0,r^\ast}$ is also doubling with the same constant as $\mu$ 
(which controls mass ratio in \eqref{5.16}), and \eqref{6.9},
\begin{equation} \label{6.11}
\W_\varphi(\mu_0^{0,r},\mu_0^{0,\rho(r)})
= \W_\varphi(\mu_1,\nu_1)
\leq C \W_\varphi(\mu_0^{0,r^\ast},\mu_0^{0,r_+})
\leq C  \alpha_{\cD}(0,r_+).
\end{equation}

\ms
In order to show that $\Tan(\mu,0)$ is one-dimensional,
we define a specific sequence of measures $\mu^{0,r_j}$, which will be used 
to approximate all the tangent measures of $\Tan(\mu,0)$ up to a multiplicative constant.
We start with $r_0 = 1/4$, and define $r_j$ by induction, taking
$r_{j+1} = \rho(r_j)$ for $j \geq 0$. 
Note that for all choice of integers $0 \leq k \leq l$,
by \eqref{5.6}, \eqref{6.11}, and \eqref{6.3} we have
\begin{eqnarray} \label{6.12}
\W_\varphi(\mu_0^{0,r_{k}},\mu_0^{0,r_{l+1}})
&\leq& \sum_{k \leq j \leq l} \W_\varphi(\mu_0^{0,r_j},\mu_0^{0,r_{j+1}})
\leq C \sum_{k \leq j < l} \alpha_{\cD}(0,(r_j)_+)
\nonumber\\
&\leq& 2C \sum_{k \leq j < l} \int_{2r_j}^{4r_j} \alpha_{\cD}(0,t) {dt\over t}.
\end{eqnarray}
Recall that $r_{j+1} = \rho(r_j) \leq \lambda_1^{-1}r_j$ (see \eqref{6.7}), 
thus the $r_j$'s decay at a definite rate. 
Therefore the intervals $[2r_j,4r_j]$ have bounded overlap, and since they are
all contained in $(0,4r_k]$, we obtain
\begin{equation} \label{6.13}
\W_\varphi(\mu_0^{0,r_{k}},\mu_0^{0,r_{l+1}})
\leq C \int_{0}^{4r_k} \alpha_{\cD}(0,t) {dt\over t}.
\end{equation}

\ms 
 Let $\sigma \in \Tan(\mu,0)$ be given. There exist sequences $\{\rho_k\}$ and $\{a_k\}$ 
 such that $\rho_k \in (0,1/4]$, $\lim _{k\to\infty}\rho_k =0$, $a_k>0$, and
\begin{equation} \label{6.14}
\sigma_k = a_k \mu^{0,\rho_k} \ \text{  converges weakly to } \sigma.
\end{equation}
Let $j = j(k)$ denote the largest integer such that 
$r_{j} \geq \rho_k$. Thus $j \geq 0$ (because $r_0 = 1/4$), and
$r_{j+1}<\rho_k$; since 
$r_{j+1} = \rho(r_j) \in [\lambda_2^{-1}r_j,\lambda_1^{-1} r_j]$
(by \eqref{6.7}), we get that
\begin{equation} \label{6.15}
\lambda_2^{-1} r_{j(k)}<\rho_k  \leq r_{j(k)}.
\end{equation}
Set $\theta_k = \rho_k /r_{j(k)}\in [\lambda_2^{-1},1]$;
we may replace $\{ \rho_k \}$ by a subsequence such that
the $\theta_k$ converge to a limit $\theta$. 
Consider the dilation $D_k$ defined by $D_k(u) = \theta_k u$,
and set $D(u) = \theta u$. Notice that by \eqref{6.14} and \eqref{1.11} 
\begin{equation} \label{6.16}
[D_k]_\sharp \sigma_k = a_k [D_k]_\sharp \mu^{0,\rho_k}
= a_k \mu^{0,\rho_k/\theta_k} = a_k \mu^{0,r_{j(k)}}.
\end{equation}
Also, the $[D_k]_\sharp \sigma_k$ converge weakly to $D_\sharp \sigma$. 
In fact for $f$ continuous and compactly supported,
\begin{eqnarray}\label{6.16A}
\lim_{k\to\infty}\int f \, d[D_k]_\sharp \sigma_k 
&= &\lim_{k\to\infty}\int f(\theta_k^{-1} x) d\sigma_k(x)
= \lim_{k\to\infty}\int f(\theta^{-1} x) d\sigma_k(x)
\nonumber\\
&= &\int f(\theta^{-1} x) d\sigma(x) = \int f \, d[D]_\sharp \sigma
\end{eqnarray}
(use the uniform continuity of $f$ and local uniform bounds on the $\sigma_k$).
By \eqref{6.1}, $\sigma$ is a flat measure; then $D_\sharp \sigma = \sigma$
and  \eqref{6.16A} shows that
\begin{equation} \label{6.17}
\{ a_k \mu^{r_{j(k)}} \} \text{ converges weakly to } \sigma.
\end{equation}
If $\sigma'$ is another nonzero element of $\Tan(\mu,0)$, we can find 
other sequences $\{ j'(k) \}$ and $\{ a'_k \}$, with
$\lim_{k \to \infty} j'(k) = \infty$ (by the analogue of \eqref{6.15}),
such that
\begin{equation} \label{6.18}
\{ a'_k \mu^{0,r_{j'(k)}} \} \text{ converges weakly to } \sigma'.
\end{equation}
By Lemma \ref{t5.4}, then \eqref{5.5}, and then \eqref{6.13} and \eqref{6.2},
\begin{eqnarray}  \label{6.19}
\W_{\varphi}(\sigma,\sigma') \hskip-0.2cm &\leq& 
\liminf_{k \to \infty} \W_{\varphi}(a_k \mu^{0,r_{j(k)}},a'_k \mu^{0,r_{j'(k)}})
\nonumber
= \liminf_{k \to \infty} \W_{\varphi}(\mu^{0,r_{j(k)}},\mu^{0,r_{j'(k)}}) 
\\
&\leq& C  \liminf_{k \to \infty} 
\int_{0}^{4 \max(r_{j(k)},r_{j'(k)})} \alpha_{\cD}(0,t) {dt\over t} = 0.
\end{eqnarray}
Then $\sigma = \sigma'$, and this completes our proof of Lemma \ref{t6.1}.
\end{proof}

\section{The decomposition of the ``good set" in $\Sigma_1$}
\label{S7}

In \eqref{1.28}-\eqref{1.29} we announced a decomposition of 
$\Sigma_0 \cap (\Sigma_1 \cup \Sigma_2)$ into pieces $\cS_d$ ($0 \leq d \leq n$),
which satisfy the property that for each $x\in\cS_d$, $\Tan(\mu,x)\subset \cF_d$. 
In this section we check that the pieces
$\Sigma_1 \cap \cS_d$ satisfy the requirements of Theorem \ref{t1.2}.
The remaining sets $\Sigma_2 \cap \cS_d$ will be treated in Section \ref{S8}.

We start with $d=0$. Set 
\begin{equation} \label{7.1}
\cS_0 = \big\{ x\in \Sigma \, ; \, \Tan(\mu,x) \subset \cF_0 \big\}.
\end{equation}
We claim (as in the statement of Theorem \ref{t1.2})
that $\cS_0$ is the set of points where $\mu$ has an atom,
and that every point of $\cS_0$ is an isolated point of $\Sigma$.

Suppose that $\mu$ has an atom at $x$. Then since $\mu$ is doubling, 
$x$ is an isolated point of $\Sigma$ (Lemma 2.3 in \cite{ADTprep}). 
We can check by hand that $\Tan(\mu,x)$ is the set $\cF_0$
of multiples of the Dirac measure at the origin, and that 
$x\in \Sigma_0 \cap \Sigma_1 \cap \Sigma_2$
(because $\alpha_\cD(x,r) = 0$ for $r$ small).

Conversely, suppose that $\Tan(\mu,x) \subset \cF_0$, and let us check
that $x$ is an isolated point of $\Sigma$. Suppose instead that we can find a
sequence $\{ x_k \}$ in $\Sigma \sm \{ x \}$ that converges to $x$.
Set $r_k = 2|x-x_k|$. Since $\mu$ is doubling, there is a subsequence of 
$\{ \mu_0^{x,r_k} \}$ which converges weakly to a measure $\sigma$.
Since $\sigma \in \Tan(\mu,x)$, $\sigma$ is a Dirac mass. Let $\zeta$
be smooth function such that $\1_ {\bB \sm B(0,1/4)}\le \zeta\le \1_{(B(0,2)\sm B(0,1/4)}$.
Then
$\int \zeta d\sigma = 0$, so 
$\lim_{k \to \infty} \int \zeta \mu_0^{x,r_k} = 0$.
On the other hand, by \eqref{1.11} and \eqref{1.2}
\begin{eqnarray}  \label{7.2}
\int \zeta d\mu_0^{x,r_k} &=& \mu(B(x,r_k))^{-1} \int \zeta d\mu^{x,r_k}
\\
&=& \mu(B(x,r_k))^{-1} \int \zeta(r_k^{-1} (y-x)) d\mu(y)
\nonumber \\
&\geq& \mu(B(x,r_k))^{-1} \mu(B(x_k,r_k/4)) \geq C_\delta^{-3}
\nonumber
\end{eqnarray}
 because $\zeta(r_k^{-1} (y-x))=1$ for 
$y\in B(x_k,r_k/4)$. This contradiction shows
that if $x\in \cS_0$, then $x$ is an isolated point in $\Sigma$, and 
then $\mu$ has a Dirac mass at $x$. 
This gives the desired description of $\cS_0$, the fact that $\cS_0$
is at most countable is easy to see. 

We may now concentrate on exponents $d \in[1,n]$. Set
\begin{equation} \label{7.3}
\cS'_d = \big\{ x\in \Sigma_0 \cap \Sigma_1 \, ; \, \Tan(\mu,x) \subset \cF_d \big\}
= \cS_d \cap \Sigma_1,
\end{equation}
where the last equality comes from \eqref{1.28}. Together with $\cS_0$,
these sets are disjoint and cover $\Sigma_1 \cap \Sigma_0$
(by \eqref{1.29}), hence also $\mu$-almost all of $\Sigma_1$
(by \eqref{1.22}). By Lemma~\ref{t6.1}, the only part of Theorem \ref{t1.2} 
concerning  $\Sigma_1$ that remains to be proven is the fact that
$\cS'_d$ is rectifiable for $1 \leq d \leq n$, and more precisely
\begin{eqnarray} \label{7.4}
&&\text{$\cS'_d$ can be covered by a countable family}
\\
&&\hskip3cm
\text{of $d$-dimensional Lipschitz graphs.}
\nonumber
\end{eqnarray}
(This is slightly more precise because we don not need to add a $\H^d$-negligible set.)
This follows from the following lemma, which is essentially known, but which 
we prove for the reader's convenience.

\begin{lemma} \label{t7.1}
Suppose that $\mu$ is a doubling measure, $\Sigma$ is its support,
$d\in \{1,...,n\}$, and  $E \subset\Sigma$ is such that for all 
$x\in E$, there is a vector space $V_x$ of dimension $d$ such that
$\Tan(x,\mu)=\big\{c\cH^{d}|_{V_{x}} \, ; \, c>0\big\}$. 
Then $E$ can be covered by a countable family of $d$-dimensional
Lipschitz graphs. 
\end{lemma}

\ms
The fact that $E = \cS'_d$ satisfies the assumption of the lemma comes from
Lemma~\ref{t6.1}.
\begin{proof}
If $d=n$, $\R^n$ is a
$d$-dimensional Lipschitz graph that covers $E$, thus
we assume that $d < n$.  We claim that 
for $x\in E$
\begin{equation} \label{7.5}
x+V_x \text{ is a tangent plane to $\Sigma$ at $x$.} 
\end{equation}
If not there is a sequence $\{ y_k \}$ in $\Sigma \sm \{ x\}$,
that tends to $x$, and such that
\begin{equation} \label{7.6}
\dist(y_k, x+V_x) \geq c |y_k -x|
\end{equation}
for some $c > 0$. Set
$r_k = 2 |y_k -x|$, and replace $\{ y_k \}$ with a subsequence for which
the $\{ \mu_0^{x,r_k} \}$ converges weakly to a measure $\sigma\in \Tan(\mu,x)$.
Let $\zeta$ be a smooth compactly supported non-negative function such that 
$\zeta(0) = 0$ on $V_x$, but 
\begin{equation} \label{7.7}
\zeta(u) = 1 \ \text{ for $u\in \bB$ such that }
\dist(u,V_x) \geq c/2.
\end{equation}
 By assumption, 
$\sigma$ is supported on $V_x$ and so $\int \zeta d\sigma = 0$.
Thus $\lim_{k \to \infty} \int \zeta \mu_0^{x,r_k} = 0$.
On the other hand, \eqref{7.6} says that for $y\in B(y_k, cr_k/4)$,
$$ 
\dist(r_k^{-1} (y-x), V_x) = r_k^{-1} \dist(y, x+V_x)
\geq r_k^{-1} [\dist(y_k, x+V_x) - {cr_k \over 4}] \geq c/4,
$$ 
hence $\zeta(r_k^{-1} (y-x))=1$ by \eqref{7.7}, and \eqref{1.11} and \eqref{1.2} imply that
\begin{eqnarray}  \label{7.8}
\int \zeta d\mu_0^{x,r_k} &=& 
\mu(B(x,r_k))^{-1} \int \zeta d\mu^{x,r_k}
\\
&=& \mu(B(x,r_k))^{-1} \int \zeta(r_k^{-1} (y-x)) d\mu(y)
\nonumber \\
&\geq& \mu(B(x,r_k))^{-1} \mu(B(x_k, c r_k/4)) \geq C^{-1}.
\nonumber 
\end{eqnarray}
This contradiction proves \eqref{7.5}.

\ms
For $\varepsilon > 0$ small and $x\in E$, choose an integer $j = j(x) \geq 0$ such that
\begin{equation} \label{7.9}
\dist(y,x+V_x) \leq \varepsilon |y-x|
\ \text{ for } y\in \Sigma \cap B(x,2^{-j(x)}).
\end{equation}
On the Grassmann manifold $G(d,n)$ of vector spaces of dimension $d$ in $\R^n$,
let us for instance use the distance defined by $\dist(V,W) = ||\pi_V - \pi_W||$, 
where $\pi_V$ and $\pi_W$ denote the orthogonal 
projections on $V$ and $W$. With this distance, $G(d,n)$ is compact.
Choose a finite family $\cV$ in $G(d,n)$ such that 
$\dist(V, \cV) \leq \varepsilon$ for $V \in G(d,n)$. Set
\begin{equation} \label{7.10}
E(V,j) = \big\{x\in E \, ; \, j(x) = j \text{ and }  \dist(V_x, V)\leq \varepsilon \big\}
\end{equation}
for $V\in \cV$ and $j \geq 0$. We now cover each $E(V,j)$
with a countable collection of $d$-dimensional Lipschitz graphs.
We claim that for each ball $B$ of radius $2^{-j-1}$,
\begin{equation} \label{7.11}
E(V,j) \cap B \ \text{ is contained in a Lipschitz graph over $V$} .
\end{equation}
Lemma \ref{t7.1} follows from this claim because the $E(V,j)$ cover $E$.
To prove the claim, let $x, y \in E(V,j)\cap B$ be given.
Observe that $|x-y| < 2^{-j}$ and $y \in E \subset \Sigma$, so
\eqref{7.9} guarantees that $\dist(y,x+V_x) \leq \varepsilon |y-x|$.
Then
$$
|\pi_V(y) - \pi_V(x)|
\leq |\pi_{V_x}(y) - \pi_{V_x}(x)| + ||\pi_V-\pi_{V_x}|| |x-y|
\leq 2\varepsilon |x-y|,
$$
which yields \eqref{7.11}. This completes our proof of Lemma \ref{t7.1}.

\end{proof}

\section{The decomposition of the ``good set" in $\Sigma_2$}
\label{S8}

Our goal in the section is to apply Theorem 1.5 in \cite{ADTprep} to the set 
$\Sigma_0 \cap \Sigma_2$, to obtain the desired decomposition. For the reader's convenience 
we include the necessary background below.

\begin{theorem} [Theorem 1.5, \cite{ADTprep}] \label{t8.1}
Let $\mu$ be a doubling measure in $\R^n$, denote by  $\Sigma$ its support,
and set
\begin{equation*}
\Sigma^0 = \big\{ x\in \Sigma \, ; \, 
\int_{0}^{1}\alpha(x,r) \frac{dr}{r} < \infty \big\},
\end{equation*} 
where
\begin{equation*}
\alpha(x,r)= \min_{d=0,1,...,n} \alpha_{d}(x,r), 
\end{equation*}
and 
\begin{equation*}
\alpha_d(x,r) = \inf\big\{ \W_1(\mu_0^{x,r},\nu_V) \, ; \, V \in A'(d,n) \big\}.
\end{equation*}
Here $A'(d,n)$ denotes the set of $n$ dimensional affine spaces which intersect $B(0,1/2)$ and 
$\nu_V = c_V \H^d\res{V} = c_V \1_{V} \H^d$, with
$  c_V = \H^d(V\cap \bB)^{-1}$.
Then there are disjoint Borel sets
$\Sigma^0(d) \subset \Sigma$, $0 \leq d \leq n$, such that 
\begin{equation*}  
\Sigma^0 = \bigcup_{d=0}^n  \Sigma^0(d),
\end{equation*}
with the following properties. 
\begin{enumerate}
\item
First, $\Sigma^0(0)$ is the set
of points of $\Sigma$ where $\mu$ has an atom; it is at most countable and
each of its point is an isolated point of $\Sigma$.
\item
For $1 \leq d \leq n$ and $x\in \Sigma^0(d)$, the limit
\begin{equation*} 
\theta_d(x) := \lim_{r \to 0} r^{-d} \mu(B(x,r))
\end{equation*}
exists, and $0 < \theta_d(x) < \infty$.
\item
For $1 \leq d \leq n$ and $x\in \Sigma^0(d)$, $\Sigma$ has a tangent $d$-plane
at $x$, $W$, and set $W^\ast = W-x$.
Then $\Tan(x,\mu)= \{c\H^{d}\res{W^\ast} \, ; \, c > 0\}$. In addition, the 
measures $\mu_0^{x,r}$ converge weakly to $\H^{d}\res{W^\ast}$.
\item Further decompose $\Sigma^0(d)$, $1 \leq d \leq n$,
into the sets
\begin{equation*}
\Sigma^0(d,k) = \big\{ x\in \Sigma^0(d) \, ; \, 
2^k \leq \theta_d(x) < 2^{k+1} \big\}, \ k\in \bZ;
\end{equation*}
then each $\Sigma^0(d,k)$ is a rectifiable set of dimension $d$, with
$\H^d(\Sigma^0(d,k) \cap B(0,R)) < \infty$ for every $R > 0$, 
$\mu$ and $\H^d$ are mutually absolutely continuous on $\Sigma^0(d,k)$, 
and $\mu = \theta_d \H^d$ there.
\end{enumerate}  
\end{theorem}

We want to apply Theorem \ref{t8.1}, so we need to show that 
for each $x\in \Sigma_0 \cap \Sigma_2$,
\begin{equation} \label{8.1}
\int_0^1 \alpha(x,t) {dt \over t} < \infty.
\end{equation}
Let $x \in \Sigma_0 \cap \Sigma_2$ be given. By
Theorem \ref{t1.3} every tangent measure $\sigma \in \Tan(\mu,x)$
is flat. To estimate to the distance from $\mu_0^{x,r}$ to 
 $\sigma$ we proceed as in Section \ref{S6}
except that we work with the whole group $\cG$ rather than $\cD$. 
We now follow that argument, without some of the details but we do emphasize the differences.

Without loss of generality, we assume that $x=0$. We use the definition of $\alpha_\cG$ and Chebyshev's
inequality to associate to each $r \in (0,1/4]$ a radius $r_+ \in [2r,4r]$ such that
\begin{equation} \label{8.2}
\alpha_{\cG}(0,r_+) \leq 2\int_{2r}^{4r} \alpha_{\cG}(0,t) {dt\over t}
\end{equation}
(see \eqref{6.3}). By the definition \eqref{1.13}, there exists $G \in \cG(0,r_+)$
 such that 
\begin{equation} \label{8.3}
\W_1(\mu_0^{G},\mu_0^{0,r_+}) \leq 2 \alpha_{\cG}(0,r_+)
\end{equation}
(see \eqref{6.4}). By \eqref{1.6} $G = \lambda R$ for
some isometry, and since $G \in \cG(x)$, \eqref{1.8} guarantee that
$G(0) = 0$ and hence $R(0) = 0$. That is, $R$ is a linear isometry.
We still have that $\lambda^{-1} = \lambda(G)^{-1} 
\in [\lambda_1r_+,\lambda_2 r_+]$, and if we set
\begin{equation} \label{8.4}
r^\ast = \lambda^{-1} \in [\lambda_1r_+,\lambda_2 r_+];
\end{equation}
as in \eqref{6.5}, we have that
\begin{equation} \label{8.5}
G(B(0,r^\ast)) = \bB
\end{equation}
and $\mu_0^{G}$ is the image of $\mu_0^{0,r^\ast}$
by a linear isometry. That is,
$\mu_0^{G} = R_\sharp \mu_0^{0,r^\ast}$
and \eqref{8.3} only yields
\begin{equation} \label{8.6}
\W_1(R_\sharp\mu_0^{0,r^\ast},\mu_0^{0,r_+}) \leq 2 \alpha_{\cG}(0,r_+)
\end{equation}
instead of \eqref{6.6}. We still multiply the radii by $r /r^\ast$, set
\begin{equation} \label{8.7}
\rho(r) =  { r  r_+  \over r^\ast}
\in [\lambda_2^{-1}r,\lambda_1^{-1} r]
\end{equation}
as in \eqref{6.7}, and deduce from \eqref{8.6} that
\begin{equation} \label{8.8}
\W_\varphi(R_\sharp\mu_0^{0,r},\mu_0^{0,\rho(r)}) \leq C \alpha_{\cG}(0,r_+),
\end{equation}
using the same proof which involves Lemma \ref{t5.1} and Lemma \ref{t5.2}
(the extra rotation does not affect the argument). Inequality \eqref{8.8} is the analogue of \eqref{6.11}.
Let us write this slightly differently. Set $R^r = R^{-1}$; then by \eqref{8.8}
\begin{equation} \label{8.9}
\W_\varphi(\mu_0^{0,r},R^r_\sharp\mu_0^{0,\rho(r)}) \leq C \alpha_{\cG}(0,r_+),
\end{equation}
since the $\W_{\varphi}$-distance 
is invariant under isometry.

\ms
Given $r_0 \leq 1/4$, we can construct a decreasing sequence
$\{ r_j \}$ as we did before, defined by $r_{j+1} = \rho(r_j)$.
Let us keep track of the rotations: set $S^0 = I$ and
$S^{j+1} =  S^j R^{r_j}$. For $ k\ge 0$ we want to estimate the numbers
\begin{equation} \label{8.10}
\delta_k = \W_\varphi(\mu_0^{0,r_0},S^{k+1}_\sharp\mu_0^{0,r_{k+1}}).
\end{equation}
Let us check by induction that
\begin{equation} \label{8.11}
\delta_k \leq C \sum_{0\leq j \leq k} \alpha_{\cG}(0,(r_j)_+).
\end{equation}
When $k=0$, this is \eqref{8.9} for $r_0$. If $k \geq 1$ and \eqref{8.11} holds for
$k-1$, the triangle inequality \eqref{5.6} yields
\begin{eqnarray}  \label{8.12}
\delta_k &\leq& \delta_{k-1} +
\W_\varphi(S^{k}_\sharp\mu_0^{0,r_{k}},S^{k+1}_\sharp\mu_0^{0,r_{k+1}})
\\
&=& \delta_{k-1} + 
\W_\varphi(S^{k}_\sharp\mu_0^{0,r_{k}},[S^{k}R^{r_k}]_\sharp\mu_0^{0,r_{k+1}})
\nonumber \\
&=& \delta_{k-1} + 
\W_\varphi(\mu_0^{0,r_{k}},R^{r_k}_\sharp\mu_0^{0,\rho(r_{k})})
\leq \delta_{k-1} + C \alpha_{\cG}(0,(r_k)_+)
\nonumber 
\end{eqnarray}
by definition of $S^{k+1}$, the invariance of $\W_\varphi$ under linear isometries, and \eqref{8.9}.
This proves \eqref{8.11}. Then \eqref{8.2} and the
same argument as in \eqref{6.12}-\eqref{6.13} yield
\begin{eqnarray} \label{8.13}
\delta_k
&\leq& C \sum_{0\leq j \leq k} \alpha_{\cG}(0,(r_j)_+)
\leq C \sum_{0\leq j \leq k} \int_{2r_j}^{4r_j} \alpha_{\cG}(0,t) {dt\over t}
\nonumber\\
&\leq&  C\int_{0}^{4r_0} \alpha_{\cG}(0,t) {dt\over t}.
\end{eqnarray}
The final integral is finite because $0\in \Sigma_2$ (see the definition \eqref{1.17}).

The measures $\mu_0^{0,r_k}$ are suitably normalized, so there is a subsequence
 which converges weakly to some measure $\sigma$
(again see Lemma 2.1 in \cite{ADTprep} for a little more detail).
There is also a further subsequence for which the $S^k$ converge
to an isometry $S$, and then the $S^k _\sharp\mu_0^{0,r_k}$ converge to 
$S_\sharp \sigma$ (proceed as for \eqref{6.16A}). 
By Lemma \ref{t5.4}, \eqref{8.10}, and \eqref{8.13},
\begin{eqnarray}  \label{8.14}
\W_\varphi(\mu_0^{0,r_0},S_\sharp \sigma) &\leq& 
\liminf_{k \to \infty} \W_\varphi(\mu_0^{0,r_0},S^{k}_\sharp\mu_0^{0,r_{k}})
=\liminf_{k \to \infty} \delta_k
\nonumber \\
&\leq&  C\int_{0}^{4r_0} \alpha_{\cG}(0,t) {dt\over t}.
\end{eqnarray}

We now use Lemma \ref{t5.3} to translate estimate \eqref{8.14} into an upper bound 
for the $\int_0^1 \alpha(x,r) \frac{dr}{r}$. 
For $t \in [1/4,1/2]$, the measure $\mu_t$ that is defined by \eqref{5.27} with $\mu$ 
replaced by $\mu_0^{0,r_0}$ is just $\mu_0^{0,tr_0}$. 
Since $\sigma$ is a flat measure, so is $S_\sharp \sigma$. 
Hence the measure $\nu_t$ built from $\nu = S_\sharp \sigma$ as in \eqref{5.27} 
is also a flat measure supported on a $d$-plane $V$ passing through the origin.
We use $\nu_t$ to estimate $\alpha(x,r)$. 
By \eqref{5.28}, the fact that $\mu$ is doubling, and \eqref{8.14}, we have
\begin{eqnarray}  \label{8.17}
\int_{1/4}^{1/2} \alpha(x,tr_0) dt
&\leq& \int_{1/4}^{1/2} \W_1(\mu_0^{0,tr_0},\nu_t) dt
=\int_{1/4}^{1/2} \W_1(\mu_t,\nu_t) dt
\\
&\leq& (8+||\varphi||_{lip}) C_\delta^2 \W_\varphi(\mu_0^{0,r_0},S_\sharp \sigma)
\leq C \int_{0}^{4r_0} \alpha_{\cG}(0,t) {dt\over t}.
\nonumber
\end{eqnarray}
Note that \eqref{8.17} holds for $r_0 \leq 1/4$.

We are now ready to prove that for $x\in \Sigma_0\cap \Sigma_2$, 
$\int_0^1 \alpha(x,r) \frac{dr}{r}<\infty$. 
Recall we are assuming $x=0$. By \eqref{8.17} and the definition \eqref{1.17} of $\Sigma_2$ we have
\begin{eqnarray}  \label{8.18}
\int_0^{1/8} \alpha(x,s) {ds \over s}
&=& \sum_{k \geq 2} \int_{2^{-k-2}}^{2^{-k-1}} \alpha(0,s) {ds \over s}
= \sum_{k \geq 2} \int_{1/4}^{1/2} \alpha(0,t2^{-k}) {dt \over t}
\nonumber \\
&\leq& 4 \sum_{k \geq 2} \int_{1/4}^{1/2} \alpha(0,t2^{-k}) dt
\nonumber\\
&\leq& 4C \sum_{k \geq 2} \int_{0}^{2^{-k+2}} \alpha_{\cG}(0,t) {dt\over t}
\\
&=& 4C \int_{0}^{1} \alpha_{\cG}(0,t) 
\Big\{ \sum_{k \geq 2} \1_{\{ k : 2^{-k+2} > t \}}(k) \Big\} {dt\over t}
\nonumber\\
&\leq& C \int_{0}^{1} \alpha_{\cG}(0,t) {\log(2/t) dt\over t} < \infty.
\nonumber
\end{eqnarray}
 
A brutal estimate shows that $\alpha(x,r) \leq 2$ for $r>0$, hence
$\int_{1/8}^1 \alpha(x,s) {ds \over s}< \infty$.
Hence the hypothesis of Theorem \ref{t8.1} hold. 
We obtain a decomposition of the set $\Sigma_0 \cap \Sigma_2$ into subsets
$\cS''_{d}$, $0 \leq d \leq n$, which satisfy all the requirements for 
Theorem \ref{t1.2}. In fact we get some additional information which
we record here. First, for every point $x\in \cS''_{d}$, $d \geq 1$,
$\Tan(\mu,x)$ is the vector space of dimension $1$ spanned by some
flat measure of dimension $d$ (see \eqref{1.27}) which implies 
\begin{equation} \label{8.19}
\cS''_{d} = \big\{ x\in \Sigma_0 \cap \Sigma_2 \, ; \,
\Tan(\mu,x) \subset \cF_d  \big\} = \Sigma_2 \cap \cS_d, 
\end{equation}
where $\cS_d$ is as in \eqref{1.28}.
Moreover once we know that $\Tan(\mu,x)$ is
the space of dimension $1$ spanned by some flat measure $\sigma \in \cF_d$,
we have that $\Sigma$ has a tangent $d$-plane at $x$ whose direction is given
by the support of $\sigma$; see \eqref{7.5}. Then $\cS''_{d}$ is rectifiable, and 
even satisfies \eqref{7.4}; see the proof above, and also\cite{ADTprep}.
This completes the proof of Theorem \ref{t1.2}.

As mentioned in Section \ref{S1} we have additional control on the size of $\mu$ on $\Sigma_2$ and the behavior of $\mu$ on $\cS''_{d}$. For $1 \leq d \leq n$ and every point 
$x\in \cS''_{d} = \Sigma_2 \cap \cS_d$, the density of $\mu$ exists, that is
\begin{equation} \label{8.20}
\theta_d(x) = \lim_{r \to 0}  r^{-d} \mu(B(x,r)) \in (0,\infty)
\end{equation}
(see \cite{ADTprep}). Morever, we have the further decomposition of $\cS''_{d}$
into sets $\cS''_{d} \cap \Sigma^0(d,k)$ where $\mu$ and $\H^d$
are mutually absolutely continuous, as in 4. in Theorem \ref{t8.1}.

\bibliographystyle{amsplain}

\providecommand{\bysame}{\leavevmode\hbox to3em{\hrulefill}\thinspace}
\providecommand{\MR}{\relax\ifhmode\unskip\space\fi MR }
\providecommand{\MRhref}[2]{
\href{http://www.ams.org/mathscinet-getitem?mr=#1}{#2}}
\providecommand{\href}[2]{#2}

\noindent {Jonas Azzam: Departament de Matem\`atiques,
Universitat Aut\`onoma de Barcelona, 08193 Bellaterra (Barcelona)
Email: jazzam@mat.uab.cat}\\

\noindent {Guy David: Universit\'e Paris-Sud, Laboratoire de Math\'{e}matiques, 
UMR 8658 Orsay, F-91405
CNRS, Orsay, F-91405. Email: guy.david@math.u-psud.fr}\\

\noindent {Tatiana Toro: University of Washington, 
Department of Mathematics,
Seattle, WA 98195-4350. Email: 
toro@math.washington.edu}

\end{document}